\documentclass[a4paper]{amsart}
\usepackage{amsmath,amssymb}
\usepackage{latexsym}
\usepackage[all]{xy}
\usepackage[latin1]{inputenc}
\usepackage{graphicx}
\usepackage{array}
\usepackage{url}
\usepackage{enumerate}
\usepackage{tikz}
\usetikzlibrary{positioning,shadows,backgrounds}

\usepackage[colorlinks=true]{hyperref}

\newcommand{\eg}{{e.g.}}
\newcommand{\ie}{{i.e.}\ }

\newcommand{\cO}{{\mathcal O}}

\newcommand{\bbZ}{{\mathbb{Z}}}

\newcommand{\half}{\frac{1}{2}}

\newcommand{\quadtext}[1]{\quad\text{#1}\quad}
\newcommand{\qquadtext}[1]{\qquad\text{#1}\qquad}

\newcommand{\restr}[1]{_{|_{\scriptstyle #1}}}

\newcommand{\equalby}[1]{\overset{\textrm{#1}}=}
\newcommand{\equalbydef}{\equalby{def}}
\newcommand{\equalbyref}[1]{\equalby{\ref{#1}}}

\newcommand{\too}{\longrightarrow}
\newcommand{\tooby}[1]{\buildrel #1\over\longrightarrow}
\newcommand{\tooo}[1]{\mathop{\vcenter{\hbox to #1em{\hrulefill}}\kern-5pt\to}}

\newcommand{\hook}{\hookrightarrow}
\newcommand{\oto}[1]{\overset{#1}\to}
\newcommand{\otoo}[1]{\overset{#1}\too}
\newcommand{\isoto}{\oto{\sim}}
\newcommand{\isotoo}{\otoo{\sim}}

\newcommand{\potimes}[1]{^{\otimes #1}}
\newcommand{\inv}{^{-1}}

\DeclareMathOperator{\im}{\mathrm{im}} 
\DeclareMathOperator{\id}{\mathrm{id}} 
\DeclareMathOperator{\coker}{\mathrm{coker}} 

\DeclareMathOperator{\Wcoh}{W}
\DeclareMathOperator{\Wper}{W}
\newcommand{\Wtot}{\Wcoh^{\text{\rm tot}}}
\newcommand{\bord}{\partial} 
\DeclareMathOperator{\ext}{e} 
\newcommand{\Gm}{\mathbb{G}_{\text{\rm m}}}

\DeclareMathOperator{\Grass}{\mathrm{Gr}}    
\DeclareMathOperator{\Pic}{\mathrm{Pic}} 

\newcommand{\can}{\omega} 

\newcommand{\BS}{X} 
\newcommand{\CS}{Y} 

\newcommand{\per}[1]{{\rm per}_{#1}} 
\newcommand{\weq}{\mathop{\,\leftrightsquigarrow\,}} 
\newcommand{\SmPic}{{\mathcal S}} 
\newcommand{\dt}[1]{\mathchoice{\mathop{\cdot}\limits_{#1}}{\cdot_{#1}}{}{}} 

\newcommand{\lbi}[1]{\underline{#1}} 

\DeclareMathOperator{\Alignname}{Al}
\newcommand{\Al}{\Alignname}
\newcommand{\Alcat}{\mathcal{A}\ell}
\DeclareMathOperator{\pullname}{Pull}
\DeclareMathOperator{\pushname}{Push}
\newcommand{\alis}[1]{{#1}{}^{{\circlearrowleft}}}
\newcommand{\alto}{\leadsto}
\newcommand{\alKto}[1]{\ensuremath\raisebox{.5ex}{$\underset{#1}{\leadsto}$}}
\newcommand{\pull}[2]{\pullname_{#2,#1}}
\newcommand{\push}[2]{\pushname_{#1,#2}}
\newcommand{\Pull}[2]{\pull{\displaystyle #1}{\displaystyle #2}}
\newcommand{\Push}[2]{\push{\displaystyle #1}{\displaystyle #2}}

\newcommand{\Set}{\mathcal{I}} 
\newcommand{\varS}{i} 
\newcommand{\FSet}{\mathcal{J}} 
\newcommand{\varF}{i} 
\newcommand{\SetZ}{\mathcal{I}} 
\newcommand{\varSZ}{i} 
\newcommand{\SetX}{\mathcal{J}} 
\newcommand{\varSX}{j} 
\newcommand{\SetU}{\mathcal{K}} 
\newcommand{\varSU}{k} 

\swapnumbers 
\theoremstyle{plain}
\newtheorem{theo}{Theorem}[section]
\newtheorem*{maintheo*}{Main Theorem}
\newtheorem{lemm}[theo]{Lemma}
\newtheorem{coro}[theo]{Corollary}
\newtheorem{prop}[theo]{Proposition}

\theoremstyle{definition}

\newtheorem{defi}[theo]{Definition}
\newtheorem{rema}[theo]{Remark}
\newtheorem{setup}[theo]{Setup}
\newtheorem{exam}[theo]{Example}
\newtheorem{nota}[theo]{Notation}

\newtheorem*{conv*}{Convention}

\title{Bases of total Witt groups and lax-similitude}
\author{Paul Balmer and Baptiste Calm{\`e}s}

\address{Paul Balmer, Department of Mathematics, UCLA, Los Angeles, CA 90095-1555, USA}
\urladdr{http://www.math.ucla.edu/\raise-.5ex\hbox{~}balmer}

\address{Baptiste Calm\`es, Universit\'e d'Artois, Laboratoire de Math\'ematiques de Lens, France}
\urladdr{http://www.math.uni-bielefeld.de/\raise-.5ex\hbox{~}bcalmes}

\thanks{The first author is supported by NSF grant DMS-0969644.}

\date{\today}

\keywords{total Witt group, triangulated category}

\subjclass{19G99, 11E81, 18E30}

\begin{document}
\bibliographystyle{amsplain}

\begin{abstract}
We explain how to work with total Witt groups, more specifically,
how to circumvent the classical embarrassment of making choices for
line bundles up to isomorphisms and up to squares.
\end{abstract}

\maketitle

\tableofcontents


\section*{Introduction}


After completing our computation of the total Witt group
$\Wtot(\Grass(d,n))$ of Grassmannians in~\cite{Balmer07_pre} and
submitting that paper for publication, a somewhat rigid anonymous
reader insisted on the problem that the total Witt group of a
scheme~$\CS$
$$
\Wtot(\CS)=\bigoplus_{[L]\in\Pic(\CS)/2}\Wcoh^*(\CS,L)
$$
did not formally exist, because the group $\Wcoh^*(\CS,L)$ depends on
a choice of a representative $L$ of $[L]\in\Pic(\CS)/2$ up to
\emph{non-unique} isomorphism. Although formally correct, it might
have occurred to a less rigid or less anonymous human being that
this problem had nothing to do with Grassmannians per se, neither
much to do with triangular Witt groups, and exists since the
foundations of Witt groups of schemes themselves, starting with Knebusch~\cite{Knebusch77b}.

For us, two alternatives presented themselves, beyond cowardly
withdrawing the paper. First, we could trace all twists in use in
the special case of Grassmannians, basically turning a reasonably
short and hopefully readable paper about Witt groups into an obscure
mess about line bundles. Or, alternatively, we could write another
paper in which we \emph{prove} that the choices do not really
matter, for Grassmannians and much more generally, providing a tool
for seven generations of Witt groupists to use happily ever after.
For some reason, we went for the second alternative. The outcome is
what the reader holds in her electronic hands.

\begin{conv*}
Throughout the paper, $\BS$ and $\CS$ stand for regular noetherian
separated schemes over~$\bbZ[\half]$, of finite Krull dimension.
(See however Remark~\ref{coherent_rema}.)
\end{conv*}

For $i\in\bbZ$ and for a line bundle~$L$ over~$\CS$, the Witt group
$\Wcoh^i_{\!Z}(\CS,L)$ is the triangular Witt group of the derived
category $\textrm{D}^{\textrm{b}}_Z(\textrm{VB}(\CS))$ of bounded
complexes of vector bundles over~$\CS$ with homology supported in a
closed subset $Z\subset \CS$, with respect to the usual duality
derived from $\mathcal{H}\textit{om}_{\cO_\CS}(-,L)$ and shifted $i$
times. See~\cite{Balmer00, Balmer01}. Of course, $\Wcoh^i(\CS,L)$
stands for $\Wcoh^i_\CS(\CS,L)$. The line bundle~$L$ is often called
\emph{the twist} of the duality.

These Witt groups are periodic in two elementary ways. First
they are $4$-periodic in shift, \ie there is an extremely natural
isomorphism
$$\Wcoh^i_{\!Z}(\CS,L) \cong \Wcoh^{i+4}_{\!Z}(\CS,L)$$
by \cite[Proposition 2.14]{Balmer00}.
On the other hand, we have a product
$$\Wcoh^{i_1}_{\!Z_1}(\CS,L_1) \times \Wcoh^{i_2}_{\!Z_2}(\CS,L_2) \to \Wcoh^{i_1+i_2}_{\!Z_1\cap Z_2}(\CS,L_1 \otimes L_2)$$
by \cite{Gille03} that recovers the usual multiplication on the
classical Witt group $\Wper(\CS)=\Wcoh^0(\CS,\cO_\CS)$ and that turns any
$\Wcoh^i_{\!Z}(\CS,L)$ into a $\Wper(\CS)$-module. For any line bundle
$M$ over~$\CS$, this product yields the other periodicity isomorphism,
called the square-periodicity isomorphism,
\begin{equation}
\label{Per_eq}%
\Wcoh^i_{\!Z}(\CS,L) \isotoo \Wcoh^i_{\!Z}(\CS, M^{\otimes 2}\otimes L)
\end{equation}
and given by multiplication with the Witt class $[M \isoto M^{\vee}
\otimes M^{\otimes 2}]\in\Wcoh^0(\CS,M^{\otimes 2})$, where $M^{\vee}$
denotes the dual of~$M$. Finally, any isomorphism $L\isoto L'$
induces isomorphisms $\Wcoh^i_{\!Z}(\CS,L) \isotoo
\Wcoh^i_{\!Z}(\CS,L')$ in the obvious way.

From these isomorphisms, it is clear that all the Witt groups of a
scheme can be recovered once we know the Witt groups
$\Wcoh^i_{\!Z}(\CS,L)$ for $L$ and $i$ varying in a set of
representatives of $\Pic(\CS)/2$ and $\bbZ/4$ respectively. The
direct sum of such groups is what is usually called ``the'' total
Witt group of the scheme~$\CS$ (with support in~$Z$). However, this
total group is not canonical since it involves the choice of a line
bundle $L$ for every class in $\Pic(\CS)/2$. Furthermore, if we want
to turn this total Witt group into a ring, using the above product,
we need to choose isomorphisms between $L_1\otimes L_2$ and the
chosen line bundle representing $[L_1\otimes L_2]$ in $\Pic(\CS)/2$,
including the choice of ``square roots'' (for the periodicity modulo
2), and so on. All these choices should further satisfy some
compatibilities, of the highest sex appeal. Last not least, it is
unclear how to make such choices in a functorial way, when varying
the scheme, under pull-back and under push-forward.

To circumvent such technical obstacles, we propose a way of keeping
the intuitive simplification allowed by the total Witt group, yet
avoiding the unpleasant use of a non-canonical object. Of course, it
is unfortunately not true in full generality that one can completely
ignore choices. Here, we provide a large class $\SmPic_\BS$ of
schemes $\CS$ over a given base~$\BS$ (Definition~\ref{defi:SmPic}),
in which it makes rigorous sense to say that a collection of Witt
classes over~$\CS$ form what we usually call a ``basis of
$\Wtot(\CS)$ over $\Wtot(\BS)$".

The initial concept is that of \emph{alignment} $A:L_1\alto L_2$
between line bundles (Section~\ref{se:align}) and the
\emph{alignment isomorphism}
$$
\alis{A} \,: \ \Wcoh^*(\CS,L_1)\isoto \Wcoh^*(\CS,L_2)
$$
induced on Witt groups (Section~\ref{se:al-W}). We show how these
interact with the two functorialities of Witt theory\,: pull-back
and conditional push-forward. This leads us to introduce lax
pull-backs and lax push-forwards (Section~\ref{se:gen-pp}) which are
heuristically pull-backs and push-forwards \emph{in which one only
cares about twists in~$\Pic/2$}. This mathematical peace of mind is
formally provided by Theorem~\ref{thm:desc}, where we show that a
change of line bundle in some $\bar \CS$ over $\CS$ can be descended
to~$\CS$, as long as both $\bar \CS$ and $\CS$ belong to our
category~$\SmPic_\BS$.

In Sections~\ref{se:mock} and~\ref{se:bases}, we generalize the
action of ``$\Wtot(\BS)$" on ``$\Wtot(\CS)$" for $\CS\in\SmPic_\BS$ by again
allowing alignments to move classes around. In this context, we
discuss the notion of ``total basis".

We then show how the fundamental geometric results, localization
long exact sequence, homotopy invariance and d\'evissage, behave
with respect to the above flexibility (Section~\ref{se:LES}) and we
explain how to trace a total basis of the total Witt group in such a
localization sequence, \emph{without} explicitly tracing the line
bundles on the nose but only their classes in $\Pic/2$. See
Theorem~\ref{thm:localization}.

We have already made use of this formalism in the revised version
of~\cite{Balmer07_pre} and we hope that other people computing
total Witt groups will enjoy the therapy.

\medskip
\section{The category of quadratic alignments}
\label{se:align}%
\medskip

%
\begin{defi}
\label{defi:align}%
Let $L_1$ and $L_2$ be line bundles on a scheme~$\CS$. We say that a
pair $A=(M,\phi)$ consisting of a line bundle $M$ on $\CS$ and an
isomorphism
$$
\phi:M\potimes{2}\otimes L_1\isotoo L_2
$$
is a \emph{(quadratic) alignment} from $L_1$ to $L_2$. Such an
alignment exists if and only if $[L_1]=[L_2]$ in $\Pic(\CS)/2$. We use
the following short notation for alignments\,:
$$
A:L_1\alto L_2\,.
$$
\end{defi}

\begin{defi}
\label{defi:al-cat}%
We say that two alignments $A=(M,\phi)$ and $A'=(M',\phi')$ from
$L_1$ to $L_2$ are \emph{isomorphic}, denoted $A\simeq A'$, if there
exists an isomorphism $\tau:M\isoto M'$ such that the following
diagram commutes\,:
$$
\xymatrix{
M\potimes{2}\otimes L_1 \ar[r]^-{\phi}
\ar[d]_-{{\tau}\potimes{2}\otimes 1}
& L_2 \ar@{=}[d]
\\
{M'}\potimes{2}\otimes L_1 \ar[r]^-{\phi'}
& L_2\,.\!\!
}
$$
There is only a \emph{set} of isomorphism classes of alignments from
$L_1$ to~$L_2$\,:
$$
\Al_{\CS}(L_1,L_2):=\{A:L_1\alto L_2\}/_{\displaystyle \simeq}\,.
$$
We denote by $[A]_\simeq$ in $\Al(L_1,L_2)$ the isomorphism class of
an alignment~$A:L_1\alto L_2$. We define the \emph{alignment
category}, denoted by
$$
\Alcat(\CS)\,,
$$
to be the category of line bundles over~$\CS$ with $\Al_{\CS}(L_1,L_2)$
as morphism sets from $L_1$ to~$L_2$. The composition of
$L_1\overset{A_1}\alto L_2\overset{A_2}\alto L_3$ is defined as
follows. If, say, $A_i=(M_i,\phi_i)$ with
$\phi_i:M_i\potimes{2}\otimes L_i\isoto L_{i+1}$ for $i=1,2$, then
$$
A_2\circ A_1:=\big(M_2\otimes M_1\,,\, \phi_3)\,:\ L_1\alto L_3
$$
where $\phi_3$ is the obvious isomorphism (essentially $\phi_2\circ
\phi_1$)
$$
(M_2\otimes M_1)\potimes{2}\otimes L_1\underset{(23)}\isotoo
{M_2}\potimes{2}\otimes M_1\potimes{2}\otimes
L_1\underset{1\otimes\phi_1}\isotoo {M_2}\potimes{2}\otimes
L_2\underset{\phi_2}\isotoo L_3\,.
$$
Composition is compatible with isomorphisms and is associative up to
isomorphism hence turns $\Alcat(\CS)$ into a category, in
which $\id_L$ is given by
$\big(\cO_\CS\,,\,\cO_\CS\potimes{2}\otimes L\cong L\big)$.
\end{defi}

\begin{lemm}
\label{lemm:groupoid}%
The category $\Alcat(\CS)$ is a groupoid, \ie for every alignment
$A:L_1\alto L_2$ there exists $A':L_2\alto L_1$ with $A\circ
A'\simeq\id_{L_2}$ and $A'\circ A\simeq \id_{L_1}$, that is,
$[A']_\simeq=[A]_\simeq\inv$. In particular, given $A_1$ and $A_2$
with same source (resp.\ same target) there always exists an
alignment $B$ such that $B\circ A_1\simeq A_2$ (resp.\ $A_1\circ
B\simeq A_2$).
\end{lemm}

\begin{proof}
If $A=(M,\phi)$ take $A'=(M\inv,\phi'):L_2\alto L_1$ where $\phi'$
is $(M\inv)\potimes{2}\otimes\phi\inv$ followed by the canonical
isomorphism $(M\inv)\potimes{2}\otimes M\potimes{2}\otimes L_1\cong
L_1$.
\end{proof}

\begin{defi}
\label{defi:al-op}%
We can also tensor alignments\,: $(L_1\overset{A_1}\alto
L_1')\otimes(L_2\overset{A_2}\alto L_2')$. If, say,
$A_i=(M_i,\phi_i)$ with $\phi_i:M_i\potimes{2}\otimes L_i\isoto
L_i'$ for $i=1,2$, then their \emph{tensor product}
$$
A_1\otimes A_2:=\big(M_1\otimes M_2\,,\phi_4\big)\,:\ L_1\otimes
L_2\alto L_1'\otimes L_2'
$$
is given by the obvious isomorphism ($\phi_4$ is morally essentially
$\phi_1\otimes\phi_2$)
$$
(M_1\otimes M_2)\potimes{2}\otimes L_1\otimes
L_2\underset{(2453)}{\overset{\cong}\too} {M_1}\potimes{2}\otimes
L_1\otimes {M_2}\potimes{2}\otimes
L_2\underset{\phi_1\otimes\phi_2}\isotoo L_1'\otimes L_2'\,.
$$
The reader can verify that this tensor product preserves
isomorphisms of alignments and turns $\Alcat(\CS)$ into a symmetric
monoidal category.
\end{defi}

\begin{lemm}
\label{lemm:invert}%
Every line bundle $L$ over $\CS$ is invertible in~$\Alcat(\CS)$ for this
$\otimes$. In particular, the map $L\otimes-$ induces a bijection
$\Al_{\CS}(L_1,L_2)\isoto \Al_{\CS}(L\otimes L_1\,,\,L\otimes L_2)$ and
similarly for $-\otimes L$.
\end{lemm}

\begin{proof}
Indeed, $L\inv$ is also the inverse of $L$ in~$\Alcat(\CS)$.
\end{proof}

Finally, the following remark contains some functoriality of alignments with respect to the
scheme~$\CS$\,:
\begin{rema}
\label{rema:al-f*}%
Given a morphism $f:\bar \CS\to \CS$ and an alignment
$A=(M,\phi):L_1\alto L_2$ on~$\CS$, there is an obvious alignment
$f^*(A):=(f^*M,f^*\phi):f^*L_1\alto f^*L_2$. The reader will verify
functoriality of this construction\,: $A\simeq A'\implies f^*A\simeq
f^*A'$ and $f^*(A_2\circ A_1)\simeq f^*(A_2)\circ f^*(A_1)$ and
$f^*(A_1\otimes A_2) \simeq f^*A_1\otimes f^*A_2$ and
$(gf)^*(A)\simeq f^*(g^*(A))$. So we get a well-defined
$\otimes$-functor
$$
f^*:\Alcat(\CS)\too\Alcat(\bar \CS)
$$
such that $(gf)^*\simeq f^*g^*$.
\smallbreak

If $\CS$ and $\bar \CS$ are regular and $f:\bar \CS\to \CS$ is
proper, with relative canonical bundle~$\can_f$\,, and if
$A=(M,\phi):\,L_1\alto L_2$ is an alignment on~$\CS$, we define an
alignment $f^!(A):\,\can_f\otimes f^*L_1\alto \can_f\otimes f^*L_2$
on~$\bar \CS$ by
\begin{equation}
\label{eq:al-f!}%
f^!(A)=\big(f^*M\,,\,f^!\phi_2\big)
\end{equation}
where $f^!\phi_2$ is the canonical isomorphism
$$
f^*M\potimes{2}\otimes\can_{f}\otimes f^*L_1\underset{(123)}\isotoo
\can_{f}\otimes f^*M\potimes{2}\otimes f^*L_1\cong \can_{f}\otimes
f^*\big(M\potimes{2}\otimes L_1\big)\underset{1\otimes
f^*\phi}\isotoo \can_f\otimes f^*L_2\,.
$$
Using the monoidal structure, this can be stated as $f^!(A)=
\id_{\can_f}\otimes f^*(A)$. In particular, $f^!$ is as functorial
as $f^*$ was, except that $f^!$ is not monoidal.
\end{rema}

\medskip
\section{Alignment isomorphisms on Witt groups}
\label{se:al-W}%
\medskip

%
\begin{nota}
\label{nota:per}%
Let $\phi:L \isoto L'$ be an isomorphism of line bundles on a
scheme~$\CS$. The isomorphism induced on Witt groups is denoted by
$$\lbi{\phi}\,: \Wcoh^i_{\!Z}(\CS,L) \isoto \Wcoh^i_{\!Z}(\CS,L')\,.$$
The square periodicity isomorphism associated to a line bundle $M$
is denoted by
$$\per{M}\,: \Wcoh^i_{\!Z}(\CS,L) \isoto \Wcoh^i_{\!Z}(\CS,M^{\otimes 2}\otimes L).\,$$
\end{nota}

\begin{rema}
\label{rema:easy}%
Here are some easy compatibilities between those isomorphisms, that
we leave to the reader. We use the obvious notation.
\begin{enumerate}[\indent(i)]
\item
\label{easy:1}%
$\lbi{\phi'}\circ\lbi{\phi} = \lbi{\phi'\circ\phi}$.
\smallbreak
\item
\label{easy:2}%
$\per{M_2}\circ\per{M_1} =  \lbi{(23)}\,\circ\per{M_2\otimes M_1}$
for $(M_2\otimes M_1)\potimes{2}\otimes
L\overset{\cong}{\underset{(23)}\to} {M_2\!}\potimes{2}\otimes
M_1\potimes{2}\otimes L$.
\smallbreak
\item
\label{easy:3}%
$\per{M}\circ\lbi{\phi}=(\lbi{1\otimes\phi})\circ\per{M}$ where
$1\otimes\phi\,:\,M\potimes{2}\otimes L_1\isoto M\potimes{2}\otimes
L_2$.
\smallbreak
\item
\label{easy:4}%
$\per{M'}=(\lbi{\tau\potimes{2}\otimes1})\circ \per{M}$, for every
isomorphism $\tau:M\isoto M'$.
\end{enumerate}
\smallbreak
It follows that any composition of $\per{M_i}$ and $\lbi{\phi_j}$
(in any order) can always be reduced to one composition of the form
$\lbi{\phi}\circ\per{M}$. This is the true reason for the notion of
alignment introduced in Section~\ref{se:align} and yields
naturally\,:
\end{rema}

\begin{defi}
\label{defi:alis}%
For every alignment $A=(M,\phi):L_1\alto L_2$
(Definition~\ref{defi:align}), \ie $\phi:M\potimes{2}\otimes
L_1\isoto L_2$, we define an isomorphism
$\alis{A}:=\lbi{\phi}\circ\per{M}$
$$
\alis{A}\,:\ \Wcoh^*_{\!Z}(\CS,L_1)\isotoo\Wcoh^*_{\!Z}(\CS,L_2)
$$
that we call the \emph{alignment isomorphism (on Witt groups)}
associated to the alignment~$A:L_1\alto L_2$.
\end{defi}

\begin{exam}
\label{ex:u}%
To a unit $u\in \cO_\CS(\CS)^\times$ we can associate two things\,: An
isomorphism $u\cdot:L\isoto L$ for any line bundle~$L$ and a Witt
class $\langle u\rangle$ in $\Wper(\CS)=\Wcoh^0(\CS,\cO_\CS)$. It is an
easy exercise to verify that multiplication by $\langle u\rangle$ on
Witt groups is the same as the alignment isomorphism
$\lbi{(u\cdot)}=\alis{A_u}$, where $A_u=\big(\cO_\CS,(u\cdot)\big)$.
\end{exam}

With this example in mind, it seems reasonable to use the following
terminology\,:
\begin{defi}
\label{defi:weq}%
Let us say that two Witt classes $w_1\in\Wcoh^{j}_Z(\CS,L_1)$ and
$w_2\in\Wcoh^{j}_Z(\CS,L_2)$ are \emph{lax-similar} if there exists an
alignment $A:L_1\alto L_2$ such that $\alis{A}(w_1)=w_2$. This is an
equivalence relation, that we denote by
$$
w_1\weq w_2\,.
$$
\end{defi}

\begin{rema}
We have $w\weq 0$ if and only if $w=0$. However, Witt classes up to
lax-similitude do not form a group, as we already know from ordinary
similitude.
\end{rema}

\begin{prop}
\label{prop:al-comp}%
The assignment $A\mapsto \alis{A}$ respects the structures of
Section~\ref{se:align}\,:
\begin{enumerate}[\indent(a)]
\item
\label{it:al-iso}%
Isomorphic alignments $A\simeq A'$ induce the \emph{same} alignment
isomorphism $\alis{A}=\alis{A'}$ on Witt groups. Hence $\alis{(-)}$
is well-defined on~$\Alcat(\CS)$.
\item
\label{it:al-comp}%
Given two alignments $L_1\overset{A_1}\alto L_2\overset{A_2}\alto
L_3$, we have
$$
\alis{(A_2\circ A_1)}=\alis{A_2}\circ \alis{A_1}
$$
on Witt groups. Hence $\alis{(-)}$ is functorial.
\item
\label{it:al-tens}%
Given two alignments $A_i\,:\ L_i\alto L_i'$ and two Witt classes
$w_i\in \Wper^{j_i}_{\!Z_i}(\CS,L_i)$, for $i=1,2$, we have
$$
\alis{(A_1\otimes A_2)}\,(w_1\cdot
w_2)=\big(\alis{A_1}(w_1)\big)\cdot \big(\alis{A_2}(w_2)\big)\,.
$$
in $\Wper^{j_1+j_2}_{\!Z_1\cap Z_2}(\CS,L_1'\otimes L_2')$. Hence
$\alis{(-)}$ is (somewhat) monoidal.
\end{enumerate}
\end{prop}

\begin{proof}
Part~(a) follows from Remark~\ref{rema:easy}~\eqref{easy:1}
and~\eqref{easy:4}. For~(b), contemplate
$$
\xymatrix@C=4ex{
\Wcoh^*_{\!Z}(\CS,L_1) \ar[r]^-{\per{M_1}} \ar[d]^-{\per{M_2\otimes
M_1}}
& \Wcoh^*_{\!Z}(\CS,M_1\potimes{2}\otimes L_1) \ar[r]^-{\lbi{\phi_1}}
 \ar[d]^-{\per{M_2}}
& \Wcoh^*_{\!Z}(\CS,L_2) \ar[d]^-{\per{M_2}}
\\
\Wcoh^*_{\!Z}(\CS,{(M_2\otimes M_1)}\potimes{2}\otimes L_1)
\ar[r]_-{\lbi{(23)}}
& \Wcoh^*_{\!Z}(\CS,M_2\potimes{2}\otimes M_1\potimes{2}\otimes L_1) \ar[r]_-{\lbi{1\otimes\phi_1}}
& \Wcoh^*_{\!Z}(\CS,M_2\potimes{2}\otimes L_2)
}
$$
This diagram commutes by Remark~\ref{rema:easy}~\eqref{easy:2}
and~\eqref{easy:3}. Post-composing the two outside compositions with
$\lbi{\phi_2}:\,\Wcoh^*_{\!Z}(\CS,{M_2}\potimes{2}\otimes L_2)\isoto
\Wcoh^*_{\!Z}(\CS,L_3)$ yields the result. Part~(c) is straightforward
from the definition of the Witt group product~\cite{Gille03}.
\end{proof}

\begin{rema}
Since $\Alcat(\CS)$ is a groupoid (Lemma~\ref{lemm:groupoid}),
Proposition~\ref{prop:al-comp}\,\eqref{it:al-comp} gives us that any
zig-zag of alignment isomorphisms (on a given scheme) can be
realized by one single alignment isomorphism. Note however that
there might be non-isomorphic alignments from $L_1$ to $L_2$, hence
possibly non-equal alignment isomorphisms on Witt groups.
\end{rema}

Compatibility of $\alis{A}$ with morphisms $f:\bar \CS\to \CS$ is the
subject of Section~\ref{se:gen-pp}.

\begin{rema}
\label{rema:al-per}%
It is easy to verify that for every alignment $A:L_1\alto L_2$ the
alignment isomorphism
$\alis{A}:\Wcoh^*_{\!Z}(\CS,L_1)\isoto\Wcoh^*_{\!Z}(\CS,L_2)$
commutes with the periodicity isomorphism $\Wcoh^*\isoto
\Wcoh^{*+4}$, the latter being simply induced on the underlying
triangulated category by the square of the shift~$\Sigma^2$. This
automorphism $\Sigma^2$ commutes with all operations involved
in~$\alis{A}$. Note that $\Sigma^2$ has no hidden sign either.
\end{rema}

\begin{rema}
\label{rema:al-les}%
Alignment isomorphisms are compatible with localization long exact
sequences. Compatibility of $\alis{A}$ with restriction to an open
subscheme is a special case of Corollary~\ref{coro:reasy} below. Let
us discuss the other two types of homomorphisms.
\begin{enumerate}[(1)]
\item
Let $Z\subset Z'$ be closed subsets of~$\CS$. Then
extension-of-support $\ext:\Wcoh^*_{\!Z}(\CS,L)\to\Wcoh^*_{\!Z'}(\CS,L)$
obviously commutes with alignment isomorphism~$\alis{A}$.
\item
For the connecting homomorphism, let $U=\CS\smallsetminus Z$ and
consider $A:L\alto L'$ over~$\CS$ and its obvious restriction
$A\restr{U}:L\restr{U}\alto L'\restr{U}$. Then the following diagram
$$
\xymatrix{\Wcoh^{j}(U,L\restr{U}) \ar[r]^-{\bord} \ar[d]_-{\alis{A\restr{U} }}
& \Wcoh^{j+1}_{\!Z}(\CS,L) \ar[d]^-{\alis{A}}
\\
\Wcoh^{j}(U,L'\restr{U}) \ar[r]^-{\bord}
& \Wcoh^{j+1}_{\!Z}(\CS,L')
}
$$
commutes. Indeed, if $A=(M,\phi)$, then
$\alis{A}=\lbi{\phi}\circ\per{M}$ and $\bord$ clearly commutes with
$\lbi{\phi}$. It is proved in~\cite{Gille03} that $\bord$ is
$\Wper^*(\CS)$-linear, which explains why it commutes with $\per{M}$.
\end{enumerate}
\end{rema}

\medskip
\section{Lax pull-back and lax push-forward}
\label{se:gen-pp}%
\medskip

We assume for simplicity (in the treatment of the push-forward) that
all schemes $\CS$, $\bar \CS$, etc., are smooth over our regular
base~$\BS$.

Given a morphism $f:\bar \CS\to \CS$ and a closed subset $Z$ of $\CS$, we
have a pull-back
$$f^*:\Wcoh^*_{\!Z}(\CS,L)\to\Wcoh^*_{f\inv(Z)}(\bar \CS,f^*L)\,.$$
On the other hand, when $f$ is proper and has constant relative
dimension~$\dim(f)$ (the latter is automatic if $\bar \CS$ is connected)
and has relative canonical bundle~$\can_f$, we have a push-forward
homomorphism for every closed subset~$\bar Z\subset \bar \CS$
$$f_*:\,\Wcoh^{*+\dim(f)}_{\bar Z}\,\big(\bar \CS,\can_f\otimes f^*L\big)\,\too\,\Wcoh^{*}_{f(\bar Z)}(\CS,L)\,.$$
\begin{rema}
\label{rema:reasy}
\noindent Continuing Remark~\ref{rema:easy}, for every morphism of
schemes $f:\bar{\CS}\to \CS$, we have naturality\,:
\begin{enumerate}[\indent(i)]
\setcounter{enumi}{4}
\item
\label{easy:5}%
 $f^* \circ \lbi{\phi} = \lbi{f^*\phi}\,\circ f^*$.
\smallbreak
\item
\label{easy:6}%
$f^* \circ \per{M} = \per{f^*M}\circ f^*$.
\end{enumerate}
\smallbreak
\noindent Finally, if $f:\bar{\CS}\to \CS$ is proper with relative
canonical bundle~$\can_f$, we have\,:
\begin{enumerate}[\indent(i)]
\setcounter{enumi}{6}
\item
\label{easy:7}%
$\lbi{\phi}\circ f_*=f_* \circ (\lbi{1 \otimes f^*\phi})$ for the
canonical $1 \otimes f^*\phi:\, \can_f \otimes f^*L_1 \isoto \can_f
\otimes f^*L_2$.
\smallbreak
\item
\label{easy:8}%
$\per{M} \circ f_* = f_* \circ \lbi{(123)}\circ\per{f^*M}$ by the
projection formula, where
$(123):(f^*M)\potimes{2}\otimes\can_f\otimes
f^*L_1\overset{\cong}\too \can_f\otimes f^*\big(M\potimes{2}\otimes
L_1\big)$ is the canonical isomorphism.
\end{enumerate}
\end{rema}

\begin{coro}
\label{coro:reasy}%
Recall Remark~\ref{rema:al-f*} for $f^*$ and $f^!$ on alignments.
For $f:\bar \CS\to \CS$ and for an alignment $A=(M,\phi):\,L_1\alto L_2$
on~$\CS$, we have on Witt groups
$$
f^*\circ \alis{A}=\alis{(f^*A)}\circ f^*\,.
$$
When $f$ is proper, we have on Witt groups
$$
\alis{A}\circ f_*=f_*\circ\alis{(f^!A)}\,.
$$
\end{coro}

\begin{proof}
This is a compact form of Remark~\ref{rema:reasy}.
\end{proof}

\begin{rema}
Given a morphism $f:\bar \CS\to \CS$, we can compose the pull-back $f^*$
with alignment isomorphisms on $\CS$ and on $\bar \CS$. By
Corollary~\ref{coro:reasy}, any such composition can be brought down
to one of the form $\alis{\bar A}\circ f^*$ for some alignment $\bar
A$ on~$\bar \CS$. Similarly, one can combine the push-forward $f_*$
with alignment isomorphisms on~$\CS$ and $\bar \CS$ and again the
alignment isomorphisms on~$\CS$ are redundant, \ie such a composition
always boils down to one of the form $f_*\circ \alis{\bar A}$ for
some alignment $\bar A$ on~$\bar \CS$. Let us give names to those
generalized pull-back and push-forward.
\end{rema}

\begin{defi}
\label{defi:gen_pull}%
Let $f:\bar{\CS}\to \CS$ be a morphism. Let $L$ and $\bar L$ be line
bundles on $\CS$ and $\bar{\CS}$ respectively and let $\bar
A:f^*L\leadsto \bar L$ be an alignment on~$\bar \CS$
(Definition~\ref{defi:align}). We define the \emph{lax pull-back
homomorphism} from $\Wcoh^*_{\!Z}(\CS,L)$ to $\Wcoh^*_{\!f\inv
(Z)}(\bar{\CS},\bar{L})$ by
$$
\pull{f}{\bar A}\,:=\alis{\bar A}\circ
f^*\,:\quad\Wcoh^*_{\!Z}(\CS,L)\too\Wcoh^*_{f\inv (Z)}(\bar \CS,f^* L)
\isotoo\Wcoh^*_{f\inv (Z)}(\bar \CS,\bar L)\,.
$$
\end{defi}

\begin{defi}
\label{defi:gen_push}%
Let $f:\bar{\CS}\to \CS$ be a proper morphism with relative canonical
bundle~$\can_f$ and relative dimension~$d=\dim(f)$. Let $L$ and
$\bar L$ be line bundles on $\CS$ and $\bar{\CS}$ respectively, and let
$\bar A:\bar L\alto\can_f\otimes f^*L$ be an alignment on~$\bar \CS$
(note $\can_f$ here). We define the \emph{lax push-forward
homomorphism} from $\Wcoh^{*+d}_{\!\bar Z}(\bar \CS,\bar{L})$ to
$\Wcoh^*_{\!f(\bar Z)}(\CS,L)$ by
$$
\push{f}{\bar A}\,:=f_*\circ\alis{\bar A}\,:\quad
\Wcoh^{*+d}_{\!\bar Z}(\bar \CS,\bar{L})\isotoo
\Wcoh^{*+d}_{\!\bar Z}(\bar \CS,\can_f\otimes f^*
L)\too\Wcoh^*_{\!f(\bar Z)}(\CS,L)\,.
$$
\end{defi}

By Proposition~\ref{prop:al-comp}~\eqref{it:al-iso}, both
$\pull{f}{\bar A}$ and $\push{f}{\bar A}$ only depend on the
isomorphism class of~$\bar A$.
\begin{prop}
\label{prop:gen-func}%
For composable $\tilde \CS\overset{g}\to \bar \CS\overset{f}\to \CS$
and for alignments $\bar A$ on $\bar \CS$ and $\tilde A$ on $\tilde
\CS$ such that $\tilde A$ and $g^*\bar A$ are composable, we have
$$\Pull{g}{\tilde A}\,\circ \,\Pull{f}{\bar
A}=\Pull{fg\,}{\tilde A\circ(g^*\!\bar A)\,}.
$$
If instead $g^!\bar A$ and $\tilde A$ are composable and if $f$ and $g$ are moreover proper, then
$$
\Push{f}{\bar A}\,\circ \,\Push{g}{\tilde A}=\Push{fg\,}{\,(g^!\bar
A)\circ\tilde A}
$$
\end{prop}

\begin{proof}
Direct by Corollary~\ref{coro:reasy},
Definitions~\ref{defi:gen_pull} and~\ref{defi:gen_push} and
Proposition~\ref{prop:al-comp}\,\eqref{it:al-comp}.
\end{proof}

\medskip
\section{Descending alignments}
\label{se:desc}%
\medskip

Given a morphism $f:\bar \CS\to \CS$, it might happen that two line
bundles $L_1$ and $L_2$ on~$\CS$ have $f^*L_1$ and $f^*L_2$ aligned
over~$\bar \CS$, in the sense of Definition~\ref{defi:align}, without
necessarily being aligned over~$\CS$. This can cause sorrow in the
taverns. We propose a simple solution, which will be convenient in
applications.

\begin{defi}
\label{defi:SmPic}%
Recall our separated, noetherian regular base scheme~$\BS$ over
$\bbZ[\half]$. Let $\CS$ be a scheme over~$\BS$. We denote by
$\pi_\CS:\CS\to \BS$ the structure morphism and by $\Pic_\BS(\CS)$ the
cokernel of $\pi_\CS^*:\Pic(\BS)\to \Pic(\CS)$. Consider the full
subcategory
$$\SmPic_\BS$$
of smooth $\BS$-schemes $\pi_\CS:\CS\to \BS$ (with morphisms
over~$\BS$ of course) such that\,:
\begin{enumerate}[\indent(I)]
\item \label{picinj_item} The map $\pi_\CS^*:\Pic(\BS) \to \Pic(\CS)$ is injective.
\smallbreak
\item \label{2torzero_item} The abelian group
$\Pic_\BS(\CS)=\Pic(\CS)/\pi_\CS^*\big(\Pic(\BS)\big)$ has no $2$-torsion.
\smallbreak
\item \label{Gmhyp_item} The map $\pi_\CS^*:\Gm(\BS) \to \Gm(\CS)/\Gm(\CS)^2$ is surjective.
\end{enumerate}
\end{defi}

\begin{rema}
Note that $\BS$ itself is in~$\SmPic_\BS$. Projective spaces over $X$, Grassmannians and various flag varieties over $X$ are in $\SmPic_X$, as explained in \cite{Balmer07_pre}. If $X$ is local (\eg\ the spectrum of a field) and thus has trivial Picard group, then projective varieties over $X$ with no $2$-torsion in their Picard group are in $\SmPic_X$. 
\end{rema}

Assumptions~\eqref{picinj_item} and~\eqref{2torzero_item}
allow an easy chase\,:
\begin{lemm}
\label{lemm:chase}%
Let $f:\bar \CS\to \CS$ be a morphism in $\SmPic_\BS$. Then\,:
\begin{enumerate}[\indent(a)]
\item\label{it:2-tor}%
The homomorphism $f^*:\Pic(\CS)\to\Pic(\bar \CS)$ induces an
isomorphism on $2$-torsion subgroups ${}_2\!\Pic(\CS)\isoto
{}_2\!\Pic(\bar \CS)$.
\smallbreak
\item\label{it:seq}%
The sequence $0\to \Pic(\BS)/2 \to \Pic(\CS)/2 \to \Pic_\BS(\CS)/2 \to 0$ is
exact.
\smallbreak
\item\label{it:desc}%
The homomorphism $\Pic(\CS)/2\too\Pic_\BS(\CS)/2\,\oplus\,\Pic(\bar \CS)/2$
is injective.
\smallbreak
\item\label{it:trivial}%
If a line bundle $\bar L$ on $\bar \CS$ is such that $[\bar L]=[f^*L]$
in $\Pic_\BS(\bar \CS)/2$ for some $L$ over~$\CS$, then there exists $L'$
over $\CS$ with $[L']=[L]$ in $\Pic_\BS(\CS)/2$ and $[\bar L]=[f^*L']$ in
$\Pic(\bar \CS)/2$ already.
\end{enumerate}\end{lemm}
\begin{proof}
Since $\pi_{\CS}\,f=\pi_{\bar \CS}$ it suffices to
prove~\eqref{it:2-tor} for $f=\pi_{\CS}:\CS\to \BS$. Multiplication
by $2$ yields an endomorphism of the following short sequence of
abelian groups
$$0 \to \Pic(\BS) \otoo{\pi_{\CS}^*} \Pic(\CS) \too \Pic_\BS(\CS) \to 0,$$
which is exact by~\eqref{picinj_item} above. The Snake Lemma and
Assumption~\eqref{2torzero_item} give~\eqref{it:2-tor}
and~\eqref{it:seq}. For~\eqref{it:desc}, it suffices to compare the
exact sequences~\eqref{it:seq} for $\CS$ and~$\bar \CS$ via~$f^*$\,:
$$
\xymatrix@R=1.5em{0 \ar[r]
& \Pic(\BS)/2 \ar[r] \ar@{=}[d]
& \Pic(\CS)/2 \ar[r] \ar[d]^-{f^*}
& \Pic_\BS(\CS)/2 \ar[r] \ar[d]^-{f^*}
& 0
\\0 \ar[r]
& \Pic(\BS)/2 \ar[r]
& \Pic(\bar \CS)/2 \ar[r]
& \Pic_\BS(\bar \CS)/2 \ar[r]
& 0
}
$$
to chase the wanted injectivity (in fact, the right-hand square is
cartesian). For~\eqref{it:trivial}, if there exists a line bundle
$K$ on~$\BS$ with $[\bar L]=[f^*L\otimes \pi_{\bar
\CS}^*K]=[f^*(L\otimes \pi_{\CS}^*K)]$ in~$\Pic(\bar \CS)/2$, take
$L':=L\otimes \pi_\CS^*K$, which is in the same class as $L$
in~$\Pic_\BS(\CS)/2$.
\end{proof}

Here is the announced descent of alignments along morphisms $\bar
\CS\to \CS$ in~$\SmPic_\BS$.
\begin{theo}
\label{thm:desc}%
Let $f:\bar{\CS}\to \CS$ be a morphism in~$\SmPic_\BS$. Let $L_1$ and
$L_2$ be line bundles on $\CS$ and let $\bar A:\,f^*L_1\alto f^*L_2$
be an alignment on~$\bar \CS$ (Definition~\ref{defi:align}). Suppose
that $[L_1]=[L_2]$ in $\Pic_\BS(\CS)/2$. Then there exists an alignment
$A:\,L_1\alto L_2$ on~$\CS$ and an isomorphism $f^*A\simeq\bar A$
on~$\bar \CS$ (see Definition~\ref{defi:al-cat} and~\ref{rema:al-f*}).
\end{theo}
\begin{proof}
Write $\bar A=(\bar M,\bar \phi)$ so that $\bar \phi$ is an
isomorphism $\bar M\potimes{2}\otimes f^*L_1\isoto f^*L_2$ on~$\bar
\CS$. In particular, in the group $\Pic(\bar \CS)/2$, we have
$[f^*L_1]=[f^*L_2]$. Since we also assume $[L_1]=[L_2]$ in
$\Pic_\BS(\CS)/2$, we get by Lemma~\ref{lemm:chase}\,\eqref{it:desc}
that $[L_1]=[L_2]$ in $\Pic(\CS)/2$. Hence there exists a line bundle
$M'$ over~$\CS$ such that ${M'}\potimes{2}\otimes L_1\simeq L_2$.
Pulling-back to~$\bar \CS$, we get $(f^*M')\potimes{2}\simeq \bar
M\potimes{2}$ over~$\bar \CS$. Hence $[(f^*M')\otimes \bar
M\inv]\in{}_2\Pic(\bar \CS)\overset{\sim}{\leftarrow}{}_2\Pic(\CS)$ by
Lemma~\ref{lemm:chase}\,\eqref{it:2-tor}. Hence there exists
$[M'']\in{}_2\Pic(\CS)$ such that $f^*(M'\otimes M'')\simeq \bar M$.
Defining $M:=M'\otimes M''$ we now have
$$
f^*M\simeq \bar M\qquadtext{and also} M\potimes{2}\otimes L_1\simeq
L_2
$$
since ${M'}\potimes{2}\otimes L_1\simeq L_2$ already and
${M''}\potimes{2}$ is trivial.

Now choose two isomorphisms $\bar\tau_0: f^*M\isoto \bar M$ and
$\phi_0: M^{\otimes 2} \otimes L_1 \isoto L_2$, which exist by the
above construction, and consider the following diagram over~$\bar \CS$
\begin{equation}
\label{eq:comp}%
\vcenter{\xymatrix@C=4em@R=2em{
(f^*M)^{\otimes 2} \otimes f^*L_1 \ar@{=}[d]_-{}
 \ar[r]^-{\bar\tau_0^{\otimes 2}\otimes 1}
& \bar M^{\otimes 2} \otimes f^*L_1 \ar[d]^{\bar\phi}
\\
f^*(M^{\otimes 2} \otimes L_1) \ar[r]^-{f^*\phi_0} & f^*L_2\,.
}}
\end{equation}
A priori, this only commutes up to a unit $t\in\Gm(\bar{\CS})$, like
every diagram of isomorphisms of line bundles. By
Assumption~\eqref{Gmhyp_item}, $t=\pi_{\bar{\CS}}^*(u)\cdot v^2$ for
some $u\in\Gm(\BS)$ and $v\in\Gm(\bar{\CS})$. Hence we can define
$\phi=\pi_{\CS}^*(u)\cdot \phi_0$ and $\bar\tau=v^{-1}\cdot
\bar\tau_0$, so that $\phi$ and $\bar\tau$ make
Diagram~\eqref{eq:comp} strictly commute (in place of $\phi_0$ and
$\bar\tau_0$ respectively, of course).

Consider the alignment $A:=(M,\phi):\,L_1\alto L_2$ on~$\CS$. The
commutativity of~\eqref{eq:comp} precisely means $f^*A\simeq \bar A$
(see Remark~\ref{rema:al-f*}).
\end{proof}

\begin{coro}
\label{coro:desc}%
Let $f:\bar{\CS}\to \CS$ be a proper morphism in~$\SmPic_\BS$. Let $L_1$
and $L_2$ be line bundles on $\CS$ and $\bar A:\,\can_f\otimes
f^*L_1\alto \can_f\otimes f^*L_2$ be an alignment on~$\bar \CS$.
Suppose that $[L_1]=[L_2]$ in $\Pic_\BS(\CS)/2$. Then there exists an
alignment $A:\,L_1\alto L_2$ on~$\CS$ and
 an isomorphism $f^!A\simeq\bar A$
on~$\bar \CS$ (Remark~\ref{rema:al-f*}).
\end{coro}

\begin{proof}
Immediate from Theorem~\ref{thm:desc} since $f^!= \can_f\otimes
f^*:\Alcat(\CS)\to \Alcat(\bar \CS)$ and since $\bar B\mapsto
\can_f\otimes\bar B$ is a bijection by Lemma~\ref{lemm:invert}.
\end{proof}

The following two results turn the above descent of line bundle
alignments into descent of alignment isomorphisms on the level of
Witt groups.

\begin{prop}
\label{prop:desc-pull}%
Let $f:\bar{\CS}\to \CS$ be a morphism of schemes in~$\SmPic_\BS$. Let
$L_1$ and $L_2$ be line bundles on $\CS$ with $[L_1]=[L_2]\in
\Pic_\BS(\CS)/2$. Let $\bar A_1:f^*L_1\alto \bar L$ and $\bar
A_2:f^*L_2\alto \bar L$ be two alignments with same target $\bar L$
on~$\bar \CS$. Then there exists an alignment $A:L_1\alto L_2$ on~$\CS$
such that the following diagram commutes
$$\xymatrix{
\Wcoh^*_Z(\CS,L_1) \ar[rd]_-{\displaystyle\pull{f}{\bar A_1}\ }
 \ar@{-->}[rr]^-{\displaystyle\alis{A}}
&& \Wcoh^*_Z(\CS,L_2) \ar[ld]^-{\ \displaystyle\pull{f}{\bar A_2}}
\\
& \Wcoh^*_Z(\bar \CS,\bar L)
}
$$
\end{prop}

\begin{proof}
By Lemma~\ref{lemm:groupoid}, there exists an alignment $\bar
A:\,f^*L_1\alto f^* L_2$ on~$\bar \CS$ such that $\bar A_2\circ \bar
A\simeq \bar A_1$. By Theorem~\ref{thm:desc}, there exists
$A:L_1\alto L_2$ such that $f^*A\simeq \bar A$. So, $\bar A_2\circ
f^*A\simeq \bar A_1$. By Propositions~\ref{prop:gen-func}
and~\ref{prop:al-comp}\,\eqref{it:al-iso}, we have $\pull{f}{\bar
A_2}\circ\alis{A}=\pull{f}{\bar A_2\circ f^*A}=\pull{f}{\bar A_1}$.
\end{proof}

\begin{prop}
\label{prop:desc-push}%
Let $f:\bar{\CS}\to \CS$ be a proper morphism in~$\SmPic_\BS$. Let $L_1$
and $L_2$ be line bundles on~$\CS$ such that
$[L_1]=[L_2]\in\Pic_\BS(\CS)/2$. Let $\bar A_1:\bar L\alto \can_f\otimes
f^*L_1$ and $\bar A_2:\bar L\alto \can_f\otimes f^*L_2$ be two
alignments from the same source $\bar L$ on~$\bar \CS$. Then there
exists an alignment $A:L_1\alto L_2$ on~$\CS$ such that the following
diagram commutes
$$\xymatrix{
& \Wcoh^*_Z(\bar \CS,\bar L) \ar[ld]_-{\displaystyle\push{f}{\bar A_1}\ }
  \ar[rd]^-{\ \displaystyle\push{f}{\bar A_2}}
\\
\Wcoh^*_Z(\CS,L_1)
 \ar@{-->}[rr]^-{\displaystyle\alis{A}}
&& \Wcoh^*_Z(\CS,L_2)\,.
}
$$
\end{prop}
\begin{proof}
By Corollary~\ref{coro:desc} for $\bar A=\bar A_2\circ\bar A_1\inv$,
there exists $A:L_1\alto L_2$ such that $f^!A\circ\bar A_1\simeq
\bar A_2$. By Proposition~\ref{prop:gen-func},
$\alis{A}\circ\push{f}{\bar A_1}=\push{f}{f^!A\circ \bar
A_1}=\push{f}{\bar A_2}$.
\end{proof}

\begin{rema}
\label{rema:feel-good}%
Combined with Corollary~\ref{coro:reasy}, the last two propositions
tell us that, for a morphism $f:\bar \CS\to \CS$ in~$\SmPic_\BS$, it is
not so important to know where a lax pull-back $\pull{f}{\bar A}$ or
a lax push-forward $\push{f}{\bar A}$ exactly lands, as long as we
keep track of classes in $\Pic_\BS(?)/2$. Different choices can always
be ``realigned".
\end{rema}

\medskip
\section{Relative alignments and lax module structure}
\label{se:mock}%
\medskip

Now that we have a stable understanding of alignments, we introduce
a relative version of this notion, allowing a line bundle on the
base~$\BS$ to intervene.

\begin{defi}
\label{defi:X-aligned}%
Let $\pi_\CS:\CS\to \BS$ be a scheme over~$\BS$. We say that two line
bundles $L_1$ and $L_2$ over~$\CS$ are \emph{(quadratically)
$\BS$-aligned} if $[L_1]=[L_2]\in\Pic_\BS(\CS)=\coker(\pi_\CS^*:\Pic(\BS)\to
\Pic(\CS)/2)$. This amounts to the existence of a line bundle $K$
over~$\BS$ and an alignment $\pi_\CS^*K\otimes L_1\alto L_2$ as in
Definition~\ref{defi:align}.

By extension, it will be very convenient to say that a Witt class
$w\in \Wcoh_{\!Z}^*(\CS,L_1)$ is \emph{$\BS$-aligned with $L_2$} when
the line bundle $L_1$ is $\BS$-aligned with~$L_2$.

Specifying a line bundle $K$ over~$\BS$, we say that $L_1$ and $L_2$
are \emph{$K$-aligned} if $\pi_\CS^*K\otimes L_1$ is aligned
with~$L_2$. Unfolding everything, this means that there exists an
alignment $A=(M,\phi):\,\pi_\CS^*K\otimes L_1\alto L_2$, \ie an
isomorphism $\phi: M\potimes{2}\otimes \pi_\CS^*K\otimes L_1\isoto
L_2$. We call $A$ a \emph{$K$-alignment of $L_1$ with $L_2$} and use
the condensed notation
$$
A:\, L_1\alKto{K} L_2\,.
$$
\end{defi}

\begin{defi}
\label{defi:mock}%
Let $A\,:\ L_1\alKto{K}L_2$ be a $K$-alignment in~$\CS$. We are going
to define a \emph{lax product} or \emph{product realigned under~$A$}
$$
-\dt{A}-\ :\ \Wcoh^{i}(\BS,K)\times \Wcoh^{j}_{\!Z}(\CS,L_1)\too
\Wcoh^{i+j}_{\!Z}(\CS,L_2)\,.
$$
By Definition~\ref{defi:alis}, there exists an alignment isomorphism
$\alis{A}:\Wcoh^*_{\!Z}(\CS,\pi_\CS^*K\otimes
L_1)\isotoo\Wcoh^*_{\!Z}(\CS,L_2)$. Then for every Witt class $\lambda
\in \Wcoh^{i}(\BS,K)$ on the base and every Witt class $w\in
\Wcoh^{j}_{\!Z}(\CS,L_1)$ on~$\CS$, we define
$$
\lambda \dt{A} w\,:=\,\alis{A}\big(\pi_\CS^*(\lambda)\cdot w\big)
$$
for the image in~$\Wcoh^{i+j}_{\!Z}(\CS,L_2)$ of the product
$\pi_\CS^*(\lambda)\cdot w$, under the alignment
isomorphism~$\alis{A}$. We call this the \emph{lax-structure of
$\Wtot(\BS)$-module on $\Wtot_{\!Z}(\CS)$}.
\end{defi}

\begin{rema}
Taking $\lambda=1\in\Wper(\BS)$, we see that
$\lambda\dt{A}-=\alis{A}(-)$. So the above homomorphisms
$\lambda\dt{A}-$ generalize the alignment isomorphisms.
\end{rema}

This action behaves nicely with respect to all possible alignment
isomorphisms\,:
\begin{lemm}
\label{lemm:mult-al}%
Let $A:L_1\alKto{K}L_2$ be a $K$-alignment so that we have the
product $\dt{A}$ of Definition~\ref{defi:mock}. Let $B:L_2\alto
L_3$, $C:J\alto K$ and $D:L_0\alto L_1$ be (plain) alignments
over~$\CS$, $\BS$ and~$\CS$ respectively. Then $E:=B\circ
A\circ\big((\pi_\CS^*C)\otimes D\big)$ is a $J$-alignment
$L_0\alKto{J}L_3$ on~$\CS$. For every $\lambda\in \Wper^i(\BS,J)$ and
$w\in \Wcoh^j_{\!Z}(\CS,L_0)$ we have
$$
\alis{B}\Big(\alis{C}(\lambda)\dt{A}\alis{D}(w)\Big) =
\lambda\dt{E}w
$$
in~$\Wcoh^{i+j}_{\!Z}(\CS,L_3)$. In words, the lax product commutes
with lax-similitude.
\end{lemm}

\begin{proof}
A direct computation\,:
\begin{equation*}
\begin{split}
\alis{B}\Big(\alis{C}(\lambda)\dt{A}\alis{D}(w)\Big) & \equalbydef
\alis{B}\circ\alis{A}\big(\pi_\CS^*(\alis{C}(\lambda))\cdot\alis{D}(w)\big)
\\
& \equalbyref{coro:reasy}
\alis{B}\circ\alis{A}\Big(\big(\alis{(\pi_\CS^*C)}(\pi_\CS^*(\lambda))\big)\cdot\alis{D}(w)\Big)
\\
& \equalby{\ref{prop:al-comp}\eqref{it:al-tens}}
\alis{B}\circ\alis{A}\circ\alis{\big((\pi_\CS^*C)\otimes
D\big)}\,(\pi_\CS^*(\lambda)\cdot w)
\\
& \equalby{\ref{prop:al-comp}\eqref{it:al-comp}}
\alis{\Big(B\circ A\circ\big((\pi_\CS^*C)\otimes
D\big)\Big)}\,(\pi_\CS^*(\lambda)\cdot w) \equalbydef
\lambda\dt{E}w\,. \qedhere
\end{split}
\end{equation*}
\end{proof}

The real question is whether this product $\lambda \dt{A} w$ depends
significantly on the $K$-alignment~$A:L_1\alto L_2$, for $L_1$ and
$L_2$ fixed. A priori, this might be the case. However, our class of
schemes~$\SmPic_\BS$ (Definition~\ref{defi:SmPic}) turns out to be
well-behaved.

\begin{lemm}
\label{lemm:move-coeff}%
Let $\CS$ be a scheme in~$\SmPic_\BS$. Let $A_i:L_1\alKto{K_i}L_2$ be
$K_i$-alignments over~$\CS$ (for the same $L_1$ and $L_2$), for
$i=1,2$. Then there exits an alignment $C:K_1\alto K_2$ on~$\BS$ such
that, for every $\lambda_1 \in \Wcoh^*(\BS,K_1)$ and every $w\in
\Wcoh^*_{\!Z}(\CS,L_1)$, we have
$$
\lambda_1 \dt{A_1} w = \lambda_2 \dt{A_2} w
$$
in $\Wcoh^*_{\!Z}(\CS,L_2)$, where
$\lambda_2=\alis{C}(\lambda_1)\in\Wcoh^*(\BS,K_2)$.
\end{lemm}

\begin{proof}
In terms of plain alignments over~$\CS$, note that
$A_1:\,\pi_\CS^*K_1\otimes L_1\alto L_2$ and $A_2:\,\pi_\CS^*K_2\otimes
L_1\alto L_2$ have the same target~$L_2$. By
Lemma~\ref{lemm:groupoid}, there exists an alignment
$A':\pi_\CS^*K_1\otimes L_1\alto \pi_\CS^*K_2\otimes L_1$ such that
$A_2\circ A'\simeq A_1$. Note that $L_1$ appears at the two ends
of~$A'$. So, by Lemma~\ref{lemm:invert}, there exists an alignment
$A'':\pi_\CS^*K_1\alto \pi_\CS^*K_2$ such that $A'\simeq A''\otimes
L_1$. Finally, by Proposition~\ref{prop:desc-pull} applied to
$f=\pi_\CS$, there exists an alignment $C:K_1\alto K_2$ such that
$A''\simeq\pi_\CS^*C$. The result follows from
Lemma~\ref{lemm:mult-al}, applied to our $C$, and to $\bar
B:=\id_{L_2}$, $\bar D:=\id_{L_1}$, $A:=A_2$, checking that $E$ is
here $A_2\circ((\pi_\CS^*C)\otimes L_1)\simeq A_2\circ (A''\otimes
L_1)\simeq A_2\circ A'\simeq A_1$.
\end{proof}

Let us a say a word about associativity of the lax-action.
\begin{lemm}
\label{lemm:assoc}%
Let $\CS$ be an $\BS$-scheme and $K_1$ and $K_2$ be line bundles
on~$\BS$. Set $K_3:=K_2\otimes K_1$. Consider $\BS$-alignments
$A_1:L_0\alKto{K_1}L_1$ and $A_2:L_1\alKto{K_2}L_2$ and
$A_3:L_0\alKto{K_3}L_2$ over~$\CS$. Then for any choice of two out
of~$A_1$, $A_2$ and $A_3$, the third can be constructed such that
the following diagram commutes in~$\Alcat(\CS)$\,:
\begin{equation}
\label{eq:assoc}%
\vcenter{\xymatrix{ \pi_\CS^*(K_2\otimes K_1)\otimes L_0
\ar@{~>}[rr]^-{1\otimes A_1}
 \ar@{~>}[rd]_-{A_3}
&& \pi_\CS^*K_2\otimes L_1 \ar@{~>}[ld]^-{A_2}
\\
& L_2
}}
\end{equation}
Then, for every $w\in\Wcoh^{j}_{\!Z}(\CS,L_0)$,
$\lambda_1\in\Wper^{i_1}(\BS,K_1)$ and
$\lambda_2\in\Wper^{i_2}(\BS,K_2)$, we have
$$
\lambda_2\dt{A_2}\big(\lambda_1\dt{A_1}w\big)=\big(\lambda_2\cdot\lambda_1\big)\dt{A_3}w
$$
in $\Wcoh^{i_1+i_2+j}_{\!Z}(\CS,L_2)$.
\end{lemm}

\begin{proof}
The first part follows from Lemmas~\ref{lemm:groupoid}
and~\ref{lemm:invert}. The rest is direct\,:
$$
\begin{array}{rlcr}
\lambda_2\dt{A_2}\big(\lambda_1\dt{A_1}w\big)
& = \alis{A_2}\big(\pi_\CS^*\lambda_2\cdot\alis{A_1}(\pi_\CS^*\lambda_1\cdot w)\big)
& \text{by definition}
\\
& = \alis{A_2}\circ \alis{(1\otimes A_1)}\big(\pi_\CS^*\lambda_2\cdot\pi_\CS^*\lambda_1\cdot w\big)
& \text{by Proposition~\ref{prop:al-comp}\,\eqref{it:al-tens}}
\\
& = \alis{A_3}(\pi_\CS^*\lambda_2\cdot\pi_\CS^*\lambda_1\cdot w)
& \text{by~\eqref{eq:assoc} and Prop.~\ref{prop:al-comp}}
\\
& = \alis{A_3}\big(\pi_\CS^*(\lambda_2\cdot\lambda_1)\cdot w\big)
& \text{$\pi^*$ is a ring morphism}
\\
& = \big(\lambda_2\cdot\lambda_1\big)\dt{A_3}w
& \kern2em\text{by definition.}
\kern.5ex\qedhere
\end{array}
$$
\end{proof}

Let us now discuss the lax-linearity of lax pull-back and lax
push-forward.
\begin{lemm}
\label{lemm:pull-coeff}%
Let $f:\bar{\CS} \to \CS$ be a morphism in~$\SmPic_\BS$ and $Z\subset \CS$
be closed. Consider two lax pull-backs
(Definition~\ref{defi:gen_pull})\,:
$$
\pull{f}{\bar A}:\Wcoh^*_{\!Z}(\CS,L)\to\Wcoh^*_{\!f\inv Z}(\bar
\CS,\bar L)
\quadtext{and}
\pull{f}{\bar B}:\Wcoh^*_{\!Z}(\CS,M)\to\Wcoh^*_{\!f\inv Z}(\bar
\CS,\bar M)
$$
for two alignments $\bar A:f^*L\alto \bar L$ and $\bar B:f^*M\alto
\bar M$ over~$\bar \CS$. Suppose that the line bundles $L$ and $M$
over $\CS$ are $\BS$-aligned, \ie $[L]=[M]$ in $\Pic_\BS(\CS)/2$.
\begin{enumerate}[\indent(a)]
\item
\label{it:pull-easy}%
For every $K$-alignment $C:L\alKto{K}M$ over~$\CS$, there exists a
$K$-alignment $\bar C:\bar L\alKto{K}\bar M$ over~$\bar \CS$ such that
for all $\lambda\in \Wper^*(\BS,K)$ and all $w\in\Wcoh^*_{\!Z}(\CS,L)$
$$
\kern6em\pull{f}{\bar B}\big(\lambda\dt{C}w)=\lambda\dt{\bar
C}\big(\pull{f}{\bar A}(w)\big)\kern4em\text{in }\Wcoh^*_{\!f\inv
Z}(\bar \CS,\bar M)\,.
$$
\smallbreak
\item
\label{it:pull-hard}%
For every $K$-alignment $\bar C:\bar L\alKto{K}\bar M$ over~$\bar
\CS$, there exists a $K$-alignment $C:L\alKto{K}M$ over~$\CS$ such that
the very same equation holds (maybe better read from right to left
this time).
\end{enumerate}
\end{lemm}
\begin{proof}
For~(a), use Lemma~\ref{lemm:groupoid} to construct $\bar C$ such
that $\bar B\circ f^*C\simeq\bar C\circ (\id\otimes\bar A)$, \ie
solve the following left-hand square\,:
$$
\xymatrix@R=1.5em{
\pi_{\bar \CS}^*K\otimes f^*L \ar@{~>}[r]^-{f^*C}
\ar@{~>}[d]_-{\id\otimes\bar A}
& f^*M \ar@{~>}[d]^-{\bar B}
\\
\pi_{\bar \CS}^*K\otimes \bar L \ar@{~>}[r]^-{\exists\,\bar C}
& \bar M
}
\kern5em
\xymatrix@R=1.5em{
\pi_{\bar \CS}^*K\otimes f^*L \ar@{~>}[r]^-{\exists\,\bar D}
\ar@{~>}[d]_-{\id\otimes\bar A}
& f^*M \ar@{~>}[d]^-{\bar B}
\\
\pi_{\bar \CS}^*K\otimes \bar L \ar@{~>}[r]^-{\bar C}
& \bar M
}
$$
For~(b), first solve the above right-hand square to find $\bar D$
and use Theorem~\ref{thm:desc} to find $C:\pi_\CS^*K\otimes L\alto M$
such that $\bar D\simeq f^*C$. In both cases we have
\begin{equation}
\label{eq:ABCD}%
\bar B\circ f^*C\simeq\bar C\circ (\id\otimes\bar A)\,.
\end{equation}
Then compute
\begin{equation*}
\begin{array}{rlc}
\pull{f}{\bar B}\big(\lambda\dt{C}w)
& = \alis{\bar B}\circ f^*\circ \alis{C}(\pi_\CS^*(\lambda)\cdot w)
& \textrm{by definition}
\\
& = \alis{\bar B}\circ \alis{(f^*C)}\circ f^*(\pi_\CS^*(\lambda)\cdot w)
& \textrm{by Corollary~\ref{coro:reasy}}
\\
& = \alis{\bar B}\circ \alis{(f^*C)}\big(\pi_{\bar \CS}^*(\lambda)\cdot f^*(w)\big)
& \textrm{$f^*$ is a ring homomorphism}
\\
& = \alis{\bar C}\circ \alis{(\id\otimes\bar A)}\big(\pi_{\bar
\CS}^*(\lambda)\cdot f^*(w)\big)
& \textrm{by~\eqref{eq:ABCD} and Proposition~\ref{prop:al-comp}}
\\
& = \alis{\bar
C}\big(\pi_{\bar \CS}^*(\lambda)\cdot (\alis{\bar A}\circ f^*(w))\big)
& \textrm{by Proposition~\ref{prop:al-comp}\,\eqref{it:al-tens}}
\\
& = \lambda\dt{\bar
C}\big(\pull{f}{\bar A}(w)\big)\,.
& \kern3em\textrm{by definition.}
\kern1.3em\qedhere
\end{array}
\end{equation*}
\end{proof}

\begin{lemm} \label{lemm:push-coeff}
Let $f:\bar \CS \to \CS$ be a proper morphism in~$\SmPic_\BS$, of constant
relative dimension, and $\bar Z\subset \bar \CS$ be closed. Consider
two lax push-forwards (Definition~\ref{defi:gen_push})\,:
$$
\push{f}{\bar A}:\Wcoh^\star_{\!\bar Z}(\bar \CS,\bar
L)\to\Wcoh^*_{\!f\bar Z}(\CS,L)
\quadtext{and}
\push{f}{\bar B}:\Wcoh^\star_{\!\bar Z}(\bar \CS,\bar
M)\to\Wcoh^*_{\!f\bar Z}(\CS,M)
$$
where $\star=*+\dim f$ for two alignments $\bar A:\bar L\alto
\can_f\otimes f^*L$ and $\bar B:\bar M\alto \can_f\otimes f^*M$
over~$\bar \CS$. Suppose that the line bundles $L$ and $M$ over
$\CS$ are $\BS$-aligned, \ie $[L]=[M]$ in $\Pic_\BS(\CS)/2$.
\begin{enumerate}
\item
For every $K$-alignment $C:L\alKto{K}M$ over~$\CS$, there exists a
$K$-alignment $\bar C:\bar L\alKto{K}\bar M$ over~$\bar \CS$ such
that for all $\bar w\in\Wcoh^{\star}_{\!\bar Z}(\bar \CS,\bar L)$
and $\lambda\in \Wper^*(\BS,K)$ we have
$$
\kern6em \lambda\dt{C}\big(\push{f}{\bar A}(\bar w)\big)
 = \push{f}{\bar B}\big(\lambda\dt{\bar C}\bar w)\kern4em\text{in
}\Wcoh^*_{\!f\bar Z}(\CS,M)\,.
$$
\smallbreak
\item
For every $K$-alignment $\bar C:\bar L\alKto{K}\bar M$ over~$\bar
\CS$, there exists a $K$-alignment $C:L\alKto{K}M$ over~$\CS$ such
that the very same property holds (read backwards).
\end{enumerate}
\end{lemm}
\begin{proof}
For~(a), use Lemma~\ref{lemm:groupoid} to construct $\bar C$ such
that the following left-hand square commutes\,:
$$
\xymatrix@R=1.5em{
\pi_{\bar \CS}^*K\otimes \bar L \ar@{~>}[r]^-{\exists\,\bar C}
\ar@{~>}[d]_-{\bar A':=(12)\circ(\id\otimes\bar A)}
& \bar M \ar@{~>}[d]^-{\bar B}
\\
\can_f\otimes \pi_{\bar \CS}^*K\otimes f^*L \ar@{~>}[r]^-{f^!C}
& \can_f\otimes f^* M
}
\kern2em
\xymatrix@R=1.5em{
\pi_{\bar \CS}^*K\otimes \bar L \ar@{~>}[r]^-{\bar C}
\ar@{~>}[d]_-{\bar A'}
& \bar M \ar@{~>}[d]^-{\bar B}
\\
\can_f\otimes \pi_{\bar \CS}^*K\otimes f^*L
\ar@{~>}[r]^-{\exists\,\bar D}
& \can_f\otimes f^* M
}
$$
For~(b), first solve the above right-hand square to find $\bar D$
and use Corollary~\ref{coro:desc} to find $C:\pi_\CS^*K\otimes L\alto
M$ such that $\bar D\simeq f^!\bar C$. In both cases we have
\begin{equation}
\label{eq:DCBA}%
\bar B\circ \bar C\simeq f^!C\circ \bar A'\,.
\end{equation}
Then compute
\begin{equation*}
\begin{array}{rl>{\hspace{-1.5ex}}c}
\push{f}{\bar B}\big(\lambda\dt{\bar C}\bar w)
& = f_*\circ \alis{\bar B}\circ \alis{\bar C}(\pi_{\bar \CS}^*(\lambda)\cdot \bar w)
& \textrm{by definition}
\\
& =  f_*\circ \alis{(f^!C)}\circ \alis{\bar A'}(\pi_{\bar \CS}^*(\lambda)\cdot \bar w)
& \textrm{by~\eqref{eq:DCBA} and Proposition~\ref{prop:al-comp}}
\\
& = \alis{C}\circ f_*\circ\alis{\bar A'}\big(\pi_{\bar \CS}^*(\lambda)\cdot \bar w\big)
& \textrm{by Corollary~\ref{coro:reasy}}
\\
& = \alis{C}\circ f_*\Big(\lbi{(12)}\big(\pi_{\bar \CS}^*(\lambda)\cdot \alis{\bar A}(\bar
w)\big)\Big)
& \textrm{by Proposition~\ref{prop:al-comp}\,\eqref{it:al-tens}}
\\
& = \alis{C}\Big(\pi_{\CS}^*(\lambda)\cdot f_*\big(\alis{\bar A}(\bar
w)\big)\Big)
& \textrm{by projection formula for $f_*$}
\\
& = \lambda\dt{C}\big(\push{f}{\bar A}(\bar w)\big)\,.
& \textrm{by definition.}
\end{array}
\end{equation*}
The permutation of line bundles $\lbi{(12)}$ in the fourth equation
is usually dropped but actually is the precise way to state the
projection formula.
\end{proof}

\begin{rema}
\label{rema:mock-les}%
Following up on Remark~\ref{rema:al-les}, it is easy to verify that
the lax module structure is compatible with the localization long
exact sequence.
\end{rema}

\medskip
\section{Total bases of the total Witt group}
\label{se:bases}%
\medskip

We want to define what should be a basis of the non-existent
total Witt group of~$\CS$ with support in~$Z$, over the similarly
evanescent total Witt group of~$\BS$. The intuitive meaning is simple.
We want every Witt class of $\Wcoh^*(\CS,L)$ to be a sum of
lax-products of elements of the basis by Witt classes over~$\BS$ (the
coefficients) and we want no linear relation among the Witt classes
in the basis, with coefficients over~$\BS$.

\begin{setup}
\label{w_i_setup}%
We will repeatedly use the following situation\,: Let
$L_1,\ldots,L_n$ be line bundles over an $\BS$-scheme $\CS$, let
$j_1,\ldots,j_n$ be integers and $Z\subset \CS$ a closed subset. We
consider Witt classes $w_1,\ldots,w_n$ where each $w_i \in
\Wcoh^{j_i}_{\!Z}(\CS,L_i)$ lives in its own Witt group of $\CS$ with
support in the common closed subset~$Z\subset \CS$. We want to make
sense of linear combinations of $w_1,\ldots,w_n$ with coefficients
in the Witt groups of~$\BS$. With the lax module structure of
Section~\ref{se:mock}, there are many ways to multiply each $w_i$ by
a coefficient $\lambda_i$ over~$\BS$. We clarify this first.
\end{setup}

\begin{defi} \label{defi:lin-comb}
Let $w_1,\ldots,w_n$ be Witt classes over~$\CS$ as in~\ref{w_i_setup}
and assume they are \emph{$\BS$-aligned} in the sense of
Definition~\ref{defi:X-aligned}, \ie $[L_1]=\cdots=[L_n]$
in~$\Pic_\BS(\CS)/2$.

A set of \emph{compatible coefficients} for $w_1\,,\ldots,w_n$
consists of two ingredients\,:
\begin{itemize}
\item Witt classes $\lambda_1
\in \Wcoh^{i_1}(X,K_1), \,\ldots\,, \lambda_n \in
\Wcoh^{i_n}(\BS,K_n)$ over~$\BS$, with the property that
$i_1+j_1=i_2+j_2=\cdots=i_n+j_n$; call this number~$k\in\bbZ$.
\smallbreak
\item
A $K_i$-alignment $C_i:L_i\alKto{K_i} L$
(Definition~\ref{defi:\BS-aligned}), for every~$i=1,\ldots,n$, for a
common line bundle~$L$ over~$\CS$; that is, a pair $C_i=(M_i,\phi_i)$
where $M_i$ is a line bundle on~$\CS$ and $\phi_i: M_i^{\otimes 2}
\otimes \pi_\CS^* K_i \otimes L_i \isotoo L$ is an isomorphism.
\end{itemize}
When $k\in\bbZ$ and the line bundle $L$ over~$\CS$ are specified in
advance, we speak of \emph{$(k,L)$-compatible coefficients}.
Naturally, we abbreviate all this by writing
that the ``$\lambda_1,\ldots,\lambda_n$ are compatible coefficients for
$w_1,\ldots,w_n$". We also use the mildly abusive notation
$\lambda_i w_i$ for $\lambda_i \dt{C_i} w_i$ when there is no risk
of confusion, but we insist that the alignments~$C_i$ come with the
coefficients $\lambda_i$'s in any case, possibly implicitly. This lax
product $\lambda_i \dt{C_i} w_i$ belongs to $\Wcoh^k(\CS,L)$. We then
define the \emph{lax linear combination} of the $w_1,\ldots,w_n$
with coefficients $\lambda_1,\ldots,\lambda_n$ as the following
element in~$\Wcoh^k_{\!Z}(\CS,L)$\,:
$$\sum \lambda_i w_i:=\sum \lambda_i\dt{C_i}w_i\,.
$$
\end{defi}

\begin{defi} \label{defi:tot-lin-indep}
Let $Z\subset \CS$ be closed and let $\Set$ be a set. A family $(w_\varS)_{\varS\in\Set}$ of Witt classes $w_\varS\in
\Wcoh^{j(\varS)}_Z(\CS,L_\varS)$ is called \emph{totally independent
over~$\BS$} if for every finite subset $\FSet$ of $\Set$ such that the $(w_\varF)_{\varF\in\FSet}$ are $\BS$-aligned and every
compatible coefficients~$(\lambda_\varF)_{\varF\in\FSet}$, the relation
$\sum_\FSet \lambda_\varF w_\varF=0$ forces all $\lambda_\varF$ to be zero.
\end{defi}

\begin{rema}
In the following definitions, we are going to use a subset $P$ of
$\Pic_\BS(\CS)/2$. The reader might want to assume at first that $P$ is
the whole~$\Pic_\BS(\CS)/2$ for this will often be the case. Allowing
other $P$'s will only become relevant when dealing with the
functorial behavior of these notions, and only in ``fringe cases".
If $[L]\in P$, we say that $L$ is \emph{$\BS$-aligned with $P$} and we
also say that every Witt class $w\in\Wcoh^*_{\!Z}(\CS,L)$ is
\emph{$\BS$-aligned with~$P$}. These conditions are empty for
$P=\Pic_\BS(\CS)$.
\end{rema}

\begin{defi}
\label{defi:tot-gen}%
Let $Z\subset \CS$ be closed and $P$ be a subset of $\Pic_\BS(\CS)/2$. Let
$(w_\varS)_{\varS\in\Set}$ be a family of Witt classes over~$\CS$ with support in~$Z$, which
are all $\BS$-aligned with~$P$. We say that $(w_\varS)_{\varS\in\Set}$ \emph{totally
generates the $P$-part of the Witt groups of $\CS$ with support
in~$Z$, over~$\BS$,} if for every line bundle $L$ over $\CS$ such that
$[L]\in P$, every integer~$k$ and every $y\in \Wcoh^k_{\!Z}(\CS,L)$,
there exists a finite subset $\FSet$ of $\Set$ such that $(w_\varF)_{\varF\in\FSet}$ are aligned with~$L$, and
$(k,L)$-compatible coefficients $(\lambda_\varF)_{\varF\in\FSet}$
over~$\BS$ such that $y=\sum_{\varF\in\FSet}\lambda_i w_i$ as in
Definition~\ref{defi:lin-comb}.
\end{defi}

\begin{defi} \label{defi:tot-basis}
Let $P\subset\Pic_\BS(\CS)/2$ and $Z\subset \CS$ closed. We say that a family
$(w_\varS)_{\varS\in\Set}$ of Witt classes $\BS$-aligned with $P$ forms a \emph{total
basis} of the $P$-part of the Witt groups of $\CS$ with support
in~$Z$, over~$\BS$, if it is totally independent
(Definition~\ref{defi:tot-lin-indep}) and totally generates
(Definition~\ref{defi:tot-gen}).
\end{defi}

\begin{exam}
For $Z=\CS=\BS$, the unit $1\in\Wcoh^0(\BS,\cO_\BS)$ is a total basis
over~$\BS$.
\end{exam}

\begin{rema}
Unlike the classical notion, total independence does not strictly imply uniqueness of coefficients in a linear combination; given totally independent classes $w_1,\ldots,w_n$, all $\BS$-aligned with
some line bundle~$L$, we could have
$\sum\lambda_i\dt{C_i}w_i=\sum\lambda_i'\dt{C_i'}w_i$ without
$\lambda_i=\lambda_i'$ for all~$i$. Equality only
follows from independence if the alignments $C_i$ and $C_i'$ are the same. However, Lemma \ref{lemm:move-coeff} tells us that if $\CS$ is in $\SmPic_\BS$, we can find an alignment $A_i:K_i\alto K'_i$ over $\BS$ for every $i$ such that $\lambda_i\dt{C_i} w_i = \lambda''_i \dt{C_i}w_i$ with $\lambda''_i=\alis{A_i}(\lambda'_i)$. Then, we must have $\lambda_i=\lambda''_i$ by total independence.
\end{rema}

Anyway, for $\BS$-schemes in our class~$\SmPic_\BS$ (Definition~\ref{defi:SmPic}),
we have the following ``classical" interpretation of a total
basis\,:
\begin{prop} \label{prop:classic-basis}
Let $\CS\in\SmPic_\BS$. Let $P\subset \Pic_\BS(\CS)/2$ be a subset,
$Z\subset \CS$ closed and let
$\left(w_\varS\in\Wcoh^{j_\varS}_Z(\CS,L_\varS)\right)_{\varS\in\Set}$
be a set of Witt classes on~$\CS$ with support in~$Z$, such that
each $[L_{\varS}]\in P$. For each $p\in P$ set
$\Set_p=\big\{\varS\in\Set\,\big|\,[L_\varS]=p\text{ in
}\Pic_\BS(\CS)/2\big\}$. Then the following properties are equivalent\,:
\begin{enumerate}[\indent(i)]
\item
The family $(w_\varS)_{\varS\in\Set}$ is a total basis of the $P$-part of the Witt groups of $\CS$
with support in~$Z$, over~$\BS$ (Definition~\ref{defi:tot-basis}).
\smallbreak
\item
For every line bundle $L$ with $[L]\in P$, every $k\in \bbZ$ and
\emph{for every choice}, for those $\varS\in \Set_{[L]}$, of a line
bundle $K_\varS$ over $\BS$ and a $K_\varS$-alignment
$C_\varS:L_\varS\alKto{K_\varS}L$, the following map is an
isomorphism
\begin{equation}
\label{eq:iso}%
\begin{array}{ccc}
\displaystyle\theta=\theta((C_\varS)_{\varS})\,:\hspace{-1ex}\
\bigoplus_{\varS\in\Set_{[L]}}\hspace{-1ex}
\Wcoh^{k-j_\varS}(\BS,K_\varS) & \isotoo & \Wcoh^k_{\!Z}(\CS,L)
\\
(x_\varS)_{\varS\in\Set_{[L]}} \kern-10em & \longmapsto & \sum
x_\varS \dt{C_\varS}w_\varS\,.
\end{array}
\end{equation}
\smallbreak
\item For every class $p\in P$ and every
$k\in\bbZ$, there exists a choice of $L\in p$ and \emph{there exists
a choice}, for each $i\in\Set_p$, of a line bundle $K_\varS$ over
$\BS$ and a $K_\varS$-alignment $C_\varS:L_\varS\alKto{K_\varS}L$
for which~\eqref{eq:iso} is an isomorphism.
\end{enumerate}
Note that $\theta$ as in~\eqref{eq:iso} is always a homomorphism of
$\Wper(\BS)$-modules by Lemma~\ref{lemm:assoc}.
\end{prop}

\begin{proof}
$(i)\Rightarrow(ii)$\,: Injectivity is straightforward from total
independence (Definition~\ref{defi:tot-lin-indep}). For
surjectivity, let $y\in \Wcoh^k_{\!Z}(\CS,L)$ and use total
generation (Definition~\ref{defi:tot-gen}) to write $y$ as
$\sum_{\varF\in\FSet}\lambda_i\dt{C_i}w_i$ for some finite subset
$\FSet\subset\Set$, some coefficients $\lambda_i\in
\Wper^{k-j_\varF}(\BS,J_\varF)$ and some $J_\varF$-alignments
$D_\varF:L_\varF\alKto{J_\varF}L$, for $\varF\in\FSet$. A priori,
$J_\varF$ might differ from $K_\varF$ and $D_\varF$ might differ
from $C_\varF$. But Lemma~\ref{lemm:move-coeff} tells us that each
$\lambda_\varF\dt{D_\varF}w_\varF=x_\varF\dt{C_\varF}w_\varF$ for a
suitable $x_\varF\in \Wper^*(\BS,K_\varF)$ lax-similar
to~$\lambda_\varF$ . Hence
$y=\sum_{\varF\in\FSet}x_\varF\dt{C_\varF}w_\varF\in\im(\theta)$.
So, $\theta$ is surjective.

$(ii)\Rightarrow(iii)$\,: Do choose $L\in p$ and $\BS$-alignments
$C_\varF:L_\varF\alKto{K_\varF}L$ for~$\varS\in\Set_p$.

$(iii)\Rightarrow(i)$\,: Total generation is immediate from
surjectivity of~$\theta$ and Lemma~\ref{lemm:move-coeff}. For total
independence, let $(w_\varF)_{\varF\in\FSet}$ be $\BS$-aligned (as
in~\ref{w_i_setup}) and let $(\lambda_\varF)_{\varF\in\FSet}$ be
$(L',k)$-compatible coefficients for some~$L'$, such that
$\sum_{\varF\in\FSet}\lambda_\varF\dt{D_\varF}w_\varF=0$ for suitable alignments~$D_\varF$.
Note that $[L']\in P$. Choose $L$, $K_\varF$ and $C_\varF$ as in~(iii) for
$p=[L']$. Choose also $A:L'\alto L$. Then
$\alis{A}\big(\sum_{\varF\in\FSet}\lambda_\varF\dt{D_\varF}w_\varF\big)=0$ as well. By
Lemmas~\ref{lemm:mult-al} and~\ref{lemm:move-coeff} again, each
$\alis{A}(\lambda_\varF\dt{D_\varF}w_\varF)=x_\varF\dt{C_\varF}w_\varF$ for some $x_\varF$ lax-similar to~$\lambda_\varF$. We then get $\theta((x_\varF)_{\varF\in\FSet})=0$ which
forces all $x_\varF=0$ by injectivity of~$\theta$. But then
$\lambda_\varF=0$ as well since alignment isomorphisms are...
isomorphisms.
\end{proof}

\begin{lemm} \label{lemm:union}
Let $P$ and $P'$ be subsets of $\Pic_\BS(\CS)/2$ and let
$(w_\varS)_{\varS\in\Set}$ (resp.\ $(w_\varS)_{\varS\in\Set'}$\kern-0.5ex) be a
totally generating family of the $P$-part (resp.\ the $P'$-part) of
the Witt groups of $\CS$ with support in $Z$, over~$\BS$. Then the
union family $(w_\varS)_{\varS\in\Set\cup\Set'}$ is a totally
generating family of the $P\cup P'$-part of the Witt groups of $\CS$
over~$\BS$. If $P$ and $P'$ are disjoint and if the families
$(w_\varS)_{\varS\in\Set}$ and $(w_\varS)_{\varS\in\Set'}$ are both
totally independent, then their union is totally independent.
\end{lemm}
\begin{proof}
Clear.
\end{proof}

\begin{rema}
\label{rema:4-per}%
Given $i\equiv j$ modulo~4, there is a canonical isomorphism
$\Wcoh^i\isoto \Wcoh^j$, given by $(\Sigma^2)^{\frac{j-i}{4}}$,
which involves no choice and no sign. Moreover, this isomorphism
commutes with every pull-back, push-forward, alignment isomorphism
and products (still no sign because $\frac{j-i}{2}$ is even). In
other words, a Witt class $w\in\Wcoh^i$ corresponds to a unique Witt
class of~$\Wcoh^j$.

If one has an $\BS$-scheme $\CS$ and a family of Witt classes
on~$\CS$, one can wonder whether the notions of total independence and
total generation (Definitions~\ref{defi:tot-lin-indep}
and~\ref{defi:tot-gen}) would be different if one identified every
Witt class $w\in \Wcoh^i$ with its image in $\Wcoh^j$, for $j\equiv
i$ modulo~4. The answer is no, as long as one does the same on~$\BS$.

Indeed,
$\Sigma^2\big(\lambda\dt{A}w\big)=\big(\Sigma^2(\lambda)\big)\dt{A}w$.
Hence every occurrence of $\Sigma^2$ on $\CS$ can be ``absorbed" in
the coefficients.

The following analogy might help the puzzled reader. If
$R=\oplus_{i\in \bbZ}R^i$ is a $\bbZ$-graded ring and
$M=\oplus_{i\in\bbZ} M^i$ is a graded $R$-module (or $R$-algebra),
and if there exists $s\in R^4$ invertible and central (nothing
special about~4, of course), then one can consider the
$\bbZ/4$-graded ring $\bar R=\oplus_{[i]\in \bbZ/4}R^i$ with $0\leq
i\leq 3$ and the graded $\bar R$-module $\bar M=\oplus_{[i]\in
\bbZ/4}M^i$ with $0\leq i\leq 3$, where a product taking values
outside of the range $0\leq i\leq 3$ is brought back in that range
by using the \emph{unique} power of $s$ which does the job. The
point is that a collection $\mathcal{M}\subset \cup_{i\in\bbZ}M^i$
of homogeneous elements in~$M$ form an $R$-basis of~$M$ if and only
if the \emph{very same} collection forms an $\bar R$-basis of~$\bar
M$ (once brought back in the range $0\leq i\leq 3$). In particular
$\bar M$ has the same dimension over~$\bar R$ as $M$ had over~$R$.
The only possible confusion would come from the perverse
contemplation of  $\bar M$ as an $R$-module.
\end{rema}

Let us now examine how these notions behave under pull-backs or 
push-forwards. The assumption about injectivity of $f^*\restr{P}$
below is the very reason we allow the flexibility of those subsets
$P\subset\Pic_\BS(\CS)$, see Remark~\ref{rema:chunk}.

\begin{prop} \label{prop:pull-basis}
Let $f: \bar{\CS} \to \CS$ be a morphism of schemes in~$\SmPic_\BS$. Let
$P$ be a subset of $\Pic_\BS(\CS)/2$ and $Z\subset \CS$ closed. Suppose
that the pull-back map $f^*|_P:P \to \Pic_\BS(\bar{\CS})/2$ is
injective, as a map of sets. Suppose also that the pull-back
$f^*:\Wcoh^k_{\!Z}(\CS,L) \to \Wcoh^k_{f\inv(Z)}(\bar{\CS},f^*L)$ is an
isomorphism for all~$L$ with $[L]\in P$ and all~$k\in\bbZ$.

Let $(w_\varS)_{\varS\in\Set}$ be a set of Witt classes
$w_\varS\in\Wcoh^{j_\varS}_{\!Z}(\CS,L_\varS)$ over~$\CS$ with
support in~$Z$, with all~$[L_\varS]\in P$. Choose for every $\varS$
an alignment $\bar A_\varS:f^*L_\varS\alto \bar L_\varS$ over~$\bar
\CS$, hence a lax pull-back $\pull{f}{\bar
A_\varS}:\Wcoh^{j_\varS}_{\!Z}(\CS,L_\varS) \to
\Wcoh^{j_\varS}_{f\inv(Z)}(\bar{\CS},\bar L_\varS)$. Let $\bar
w_\varS:=\pull{f}{\bar A_\varS}(w_\varS)$.

Then $(w_\varS)_{\varS\in\Set}$ is a total basis of the $P$-part of
the Witt group of~$\CS$ with support in~$Z$, over~$\BS$, if and only
if $(\bar w_\varS)_{\varS\in\Set}$ is a total basis of the $f^*
P$-part of the Witt group of~$\bar \CS$ with support in~$f\inv Z$,
over~$\BS$.
\end{prop}

\begin{proof}
We simply use Proposition~\ref{prop:classic-basis}, and its
notation, both for $\CS$ and for~$\bar \CS$, in the following
commutative diagram (for line bundles and alignments to be
specified)\,:
\begin{equation}
\label{eq:theta-pull}%
\vcenter{\xy 0;<1ex,0ex>:
(0,-15)*\xybox{};
(0,0)*\xybox{\xymatrix@R=5ex{\Wper^{k-j_\varS}(\BS,K_\varS) \ar[r]^-{\theta} \ar@{=}@<1em>[d]
& \Wcoh^k_{\!Z}(\CS,L) \ar[d]^-{\pull{f}{\bar B}}_-{\simeq}
\\
\Wper^{k-j_\varS}(\BS,K_\varS) \ar[r]^-{\bar \theta}
& \Wcoh^k_{\!f\inv Z}(\bar \CS,\bar L)
}};
(-12,-0.7)*\xybox{%
$\displaystyle\bigoplus_{\varS\in\Set_{[L]}}^{\vphantom{I^I}}$
};
(-12,-10)*\xybox{%
$\displaystyle\bigoplus_{\varS\in\Set_{[\bar L]}}$
};
\endxy}
\end{equation}
Let $k\in \bbZ$. Given a line bundle $L$ on $Y$, we can set $\bar L=f^*L$ and $\bar B=\id$. Conversely, given $\bar L$ on $\bar Y$ with  $[\bar L]\in f^*(P)\subset\Pic_\BS(\CS)/2$, Lemma~\ref{lemm:chase}~\eqref{it:trivial} provides an $L$ over $\CS$
with $[L]\in P$ such that $f^*[L]=[\bar L]$ in~$\Pic(\bar \CS)/2$ already. The latter allows us to choose $\bar B:f^*L\alto \bar L$ a
(plain) alignment over~$\bar \CS$, hence to use the lax
pull-back~$\pull{f}{\bar B}$ as on the right-hand side
of~\eqref{eq:theta-pull}.

Of course, every $\bar w_\varS$ is $\BS$-aligned with $f^*P$.
Furthermore, our assumption about the Picard-group $f^*$ being
injective on~$P$ implies that $w_\varS$ is $X$-aligned with $L$ if
and only if $\bar w_\varS$ is $X$-aligned with $\bar L$ (use part
\eqref{it:desc} of \ref{lemm:chase}), thus $f^*:\Set_{[L]}\isoto
\Set_{[\bar L]}$ is a bijection and the left hand side
of~\eqref{eq:theta-pull} also makes sense.

Now choose for every $\varS\in\Set_{[L]}$ a $K_\varS$-alignment
$C_\varS:L_\varS\alKto{K_\varS}L$, so that we can create
$\theta:\,(x_\varS)\mapsto
\sum_{\varS\in\Set}x_\varS\dt{C_\varS}w_\varS$
in~\eqref{eq:theta-pull} as we did in~\eqref{eq:iso}. By
Lemma~\ref{lemm:pull-coeff}\,\eqref{it:pull-easy}, there exists
$K_\varS$-alignments $\bar C_\varS:f^*L_\varS\alKto{K_\varS}\bar L$
over~$\bar \CS$ such that
$$
\pull{f}{\bar B}(x\dt{C_\varS}w_\varS)=x\dt{\bar C_\varS}\pull{f}{\bar A_\varS}(w_\varS)
$$
for all $x\in\Wper^*(\BS,K_\varS)$ and all $\varS\in\Set_{[L]}$. So
we can define $\bar \theta$ by $(x_\varS)\mapsto\sum x_\varS\dt{\bar
C_\varS}\bar w_\varS$, to make~\eqref{eq:theta-pull} commutative.
Consequently, $\theta$ and $\bar\theta$ are simultaneously
isomorphisms. By Proposition~\ref{prop:classic-basis},
$(w_\varS)_{\varS\in\Set}$ and $(\bar w_\varS)_{\varS\in\Set}$ are
simultaneously bases.
\end{proof}

\begin{coro}
\label{coro:A1}%
Hypotheses of Proposition~\ref{prop:pull-basis} hold when $f:\bar{\CS}
\to \CS$ is an affine bundle. So, in that case, a family is a total
basis over~$\BS$ of the $P$-part of the Witt groups of $\CS$ with
support in~$Z$, if and only if, it is pulled-back to a total basis
over~$\BS$ of the $f^*(P)$-part of the Witt groups of~$\bar{\CS}$ with
support in~$f\inv Z$. \qed
\end{coro}

\begin{coro}
\label{coro:lax-basis}%
For $\CS\in\SmPic_\BS$, the notions of total independence, total
generation, and total basis are stable under alignment isomorphisms
(Definition~\ref{defi:alis}). For instance, if $(w_\varS)_{\varS\in\Set}$ is a total basis of the $P$-part of the Witt group of $\CS$ with support in~$Z$,
over~$\BS$, the family $(\alis{A_\varS}(w_\varS))_{\varS\in\Set}$ is still such a basis for any family of alignment isomorphisms $(\alis{A_\varS})_{\varS\in\Set}$ (\eg\ multiplications by a unit of~$\CS$, see Example~\ref{ex:u}).
\end{coro}

\begin{proof}
Apply Proposition~\ref{prop:pull-basis} to $f=\id_\CS$.
\end{proof}

\begin{prop} \label{prop:push-basis}
Let $f$ be a proper morphism of schemes in $\SmPic_\BS$ with
constant relative dimension~$d$. Let $P$ be a subset of
$\Pic_\BS(\CS)/2$ and $\bar Z\subset \bar \CS$. Let $f^!P:=[\can_f]
\cdot f^*P \subseteq \Pic_\BS(\bar{\CS})/2$. Suppose the
function $f^*_{|P}:P \to \Pic_\BS(\bar{\CS})/2$ injective.

Suppose also that for any line bundle $L$ such that $[L]\in P$, the
push-forward map $f_*:\Wcoh^{k+d}_{\!\bar Z}(\bar{\CS},\can_f
\otimes f^*L) \to \Wcoh^k_{\!f\bar Z}(\CS,L)$ is an isomorphism for
all~$k\in\bbZ$.

Then a family $(\bar w_\varS)_{\varS\in\Set}$ of Witt classes
on~$\bar \CS$ with support in~$\bar Z$, $\BS$-aligned with~$f^!P$\!,
is a total basis of the $f^!P$-part of the Witt group of~$\bar \CS$
with support in~$\bar Z$ if and only if the image family
$\big(\push{f}{\bar A_\varS}(\bar w_\varS)\big)_{\varS\in\Set}$
under any family of lax push-forwards corresponding to alignments
$(\bar A_\varS)_{\varS\in\Set}$ is a total basis of the $P$-part of
the Witt group of~$\CS$ with support in~$f(Z)$.
\end{prop}
\begin{proof}
The proof is similar to that of Proposition~\ref{prop:pull-basis},
mutatis mutandis. One compares two $\theta$ homomorphisms that are
``arranged" via Lemma~\ref{lemm:push-coeff} this time.
\end{proof}

In particular, the d\'evissage isomorphism for Witt groups yields the following.
\begin{coro}
\label{coro:push}%
Let $\iota:Z\hook \CS$ be a closed immersion of constant
codimension, with $Z$ and $\CS$ in~$\SmPic_\BS$. Let $P$ be a subset
of $\Pic_\BS(\CS)/2$ such that the map of sets $\iota^*_{|P}:P \to
\Pic_\BS(Z)/2$ is injective. Let $\iota^!P=[\can_\iota]
\cdot\iota^*P \subseteq \Pic_\BS(Z)/2$. Let
$(w_\varS)_{\varS\in\Set}$ be elements of the $P$-part of the total
Witt groups of $\CS$ with support in $Z$, and for each
$\varS\in\Set$, let $v_\varS$ be a Witt class in the $\iota^!P$-part
of the Witt groups of $Z$ over $\BS$ such that
$w_\varS=\iota_*(v_\varS)$ (this is always possible by d\'evissage).
The family $(v_\varS)_{\varS\in\Set}$ is a total basis of the
$\iota^!P$-part of the Witt groups of $Z$, over~$\BS$, if and only
if the family $(w_\varS)_{\varS\in\Set}$ is a total basis of the
$P$-part of the Witt groups of\; $\CS$ with support in $Z$,
over~$\BS$.
\end{coro}

\begin{rema} \label{rema:chunk}
The injectivity condition on $f^*:P \to \Pic_\BS(\bar{\CS})/2$ is
not really harmful because one can always split $\Pic_\BS(\CS)/2$ in
smaller $P$-chunks to ensure that the condition holds for each of
them, and then use Lemma~\ref{lemm:union} to obtain a total basis
for the whole $P=\Pic_\BS(\CS)/2$. This happens in ``fringe" cases,
in the cellular decomposition of the Grassmannians, for instance,
see~\cite{Balmer07_pre}.
\end{rema}

\medskip
\section{Total bases in the localization long exact sequence}
\label{se:LES}%
\medskip

Let $U$ be the open complement of
a closed subset $Z\subset \CS$, and let $\upsilon: U \hook \CS$ be the
corresponding open embedding. Assume both $\CS$ and $U$ are in~$\SmPic_\BS$. 
Let $\ext: \Wcoh^*_{\!Z}(\CS,L) \to \Wcoh^*(\CS,L)$ be the extension of support map. Recall that in this
situation, there is a long exact sequence of localization
$$
\cdots \too \Wcoh^i_{\!Z}(\CS,L) \tooby{\ext} \Wcoh^i(\CS,L)
\tooby{\upsilon^*} \Wcoh^i(U,\upsilon^*L)\tooby{\bord}
\Wcoh^{i+1}_{\!Z}(\CS,L) \to \cdots
$$

\begin{theo}
\label{thm:localization}%
Let $P$ be a subset of $\Pic_\BS(\CS)/2$. Assume that the restriction
$\upsilon^*_{|P}:P \to \Pic_\BS(U)/2$ is injective and let
$P_U=\upsilon^*(P) \subset \Pic_\BS(U)/2$.

Let $\SetZ$, $\SetX$ and $\SetU$ be sets and let $(w'_i)_{\varSZ\in
\SetZ}$ and $(w_\varSX)_{\varSX\in \SetX}$ be elements in Witt
groups of\,~$\CS$, let $(v_\varSZ)_{\varSZ\in \SetZ}$ and
$(v'_\varSU)_{\varSU\in \SetU}$ be elements in Witt groups
of\,~$\CS$ with support in $Z$ and let $(u_\varSU)_{\varSU\in
\SetU}$ and $(u'_\varSX)_{\varSX\in \SetX}$ be elements in Witt
groups of\,~$U$, whose line bundles are restricted from $\CS$. (Recall lax-similitude $\weq$ from
Definition~\ref{defi:weq}.) Suppose the following conditions hold
(see Figure~\ref{fig:uvw})\,:
\begin{enumerate}[\indent(a)]
\item \label{eweq_item} for every $\varSZ\in \SetZ$, we have $\ext(v_\varSZ)\weq w'_\varSZ$
\item \label{upsilonweq_item} for every $\varSX\in \SetX$, we have $\upsilon^*(w_\varSX)\weq u'_\varSX$
\item \label{bordweq_item} for every $\varSU\in \SetU$, we have $\bord(u_\varSU)\weq v'_\varSU$.
\end{enumerate}
Then, the following properties are satisfied\,:
\begin{enumerate}[\indent(1)]
\item \label{upsilonai_item} for every $\varSZ\in \SetZ$, we have $\upsilon^*(w'_\varSZ)=0$;
\item \label{bordci_item} for every $\varSX\in \SetX$, we have $\bord(u'_\varSX)=0$;
\item \label{ebi_item} for every $\varSU\in \SetU$, we have $\ext(v'_\varSU)=0$.
\item \label{basisinduction_item} If, furthermore, out of the three following statements:
\begin{enumerate}[\indent(i)]
\item the $(v_\varSZ)_{\varSZ\in \SetZ}$ and $(v'_\varSU)_{\varSU\in \SetU}$ form a total basis of the $P$-part of the Witt groups of\,~$\CS$ with support in $Z$, over~$\BS$,
\item the $(w'_\varSZ)_{\varSZ\in \SetZ}$ and $(w_\varSX)_{\varSX\in \SetX}$ form a total basis of the $P$-part of the Witt groups of\,~$\CS$, over~$\BS$,
\item the $(u_\varSU)_{\varSU\in \SetU}$ and $(u'_\varSX)_{\varSX\in \SetX}$ form a total basis of the $P_U$-part of the Witt groups of\,~$U$, over~$\BS$,
\end{enumerate}
two are true, then the remaining one is also true.
\end{enumerate}
\begin{figure}[h!]
\begin{tikzpicture}[>=stealth,shorten >=1pt,shorten <=1pt,text height=1.5ex,text depth=.25ex,on grid,semithick,mynode/.style={circle,fill=black!20,very thick,draw=black!80,text=black}]
\matrix[row sep=5ex,column sep=10ex] { \node (b)  [mynode] {$v_\varSZ$};
& \node (a)  [mynode] {$w'_\varSZ$}; &
\node (c)  [mynode] {$u_\varSU$}; \\
\coordinate (x);  & \coordinate (y); & \coordinate (z); \\
\node (b') [mynode] {$v'_\varSU$}; & \node (w') [mynode] {$w_\varSX$}; &
\node (c') [mynode] {$u'_\varSX$}; \\
};
\path (b)++(0.75,0.75) coordinate (b+);
\path (a)++(0.75,0.75) coordinate (a+);
\path (c)++(0.75,0.75) coordinate (c+);
\begin{pgfonlayer}{background}
\draw[fill=black!10,fill opacity=0.5,rounded corners]
(b')+(-0.75,-0.75) rectangle (b+); \draw[fill=black!10,fill
opacity=0.5,rounded corners] (w')+(-0.75,-0.75) rectangle (a+);
\draw[fill=black!10,fill opacity=0.5,rounded corners]
(c')+(-0.75,-0.75) rectangle (c+);
\end{pgfonlayer}
\draw[very thick,->] (b) -- node[above]{$e$} (a); \draw[very
thick,->] (w') -- node[above]{$\upsilon^*$}(c'); \path
(x)++(-0.75,0) coordinate (x-); \path (x)++(0.75,0) coordinate (x+);
\path (y)++(-0.75,0) coordinate (y-); \path (y)++(0.75,0) coordinate
(y+); \path (z)++(-0.75,0) coordinate (z-); \path (z)++(0.75,0)
coordinate (z+); \draw[very thick] (c) .. controls +(3,0) and +(1,0)
.. node[above right] {$\bord$} (z+); \draw[very thick,dashed] (z+)
-- (z-); \draw[very thick] (z-) -- (y+); \draw[very thick,dashed]
(y+) -- (y-); \draw[very thick] (y-) -- (x+); \draw[very
thick,dashed] (x+) -- (x-); \draw[very thick,->] (x-) .. controls
+(-1.5,0) and +(-1.5,0) .. (b'); \draw (b')+(0,-1) node (pz)
{$P$-part}; \draw (pz)+(0,-2.5ex) node {of $\Wcoh_{\!Z}(\CS)$};
\draw (w')+(0,-1) node (px) {$P$-part}; \draw (px)+(0,-2.5ex) node
{of $\Wcoh(\CS)$}; \draw (c')+(0,-1) node (pu) {$P_U$-part}; \draw
(pu)+(0,-2.5ex) node {of $\Wcoh(U)$};
\end{tikzpicture}
\caption{\label{fig:uvw}%
Families mapping to each other up to lax-similitude in
Theorem~\ref{thm:localization}. No arrow means mapped to zero.}
\end{figure}
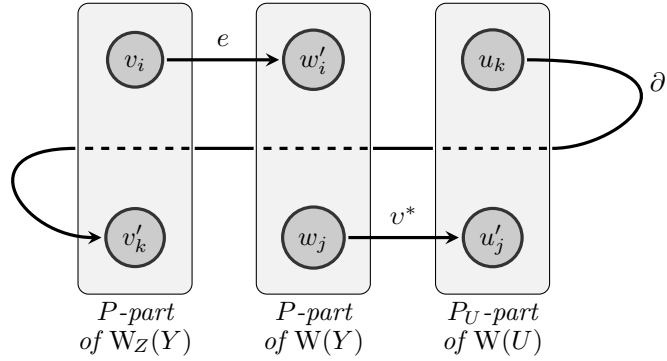
\end{theo}
\begin{proof}
Parts~\eqref{upsilonai_item}--\eqref{ebi_item} follow immediately
from Remark~\ref{rema:al-les}. The proof of
\eqref{basisinduction_item} goes through as for classical modules
over a ring. Choose a class $p\in P$ over~$\CS$, which by hypothesis
is the same thing as choosing its image $f^*(p)\in f^* P$, \ie a
class in $f^*P$ over~$U$. Up to replacing the $u'_\varSX$,
$v'_\varSU$ and $w'_\varSZ$ up to lax-similitude, which does not
change their total-basis qualities by
Corollary~\ref{coro:lax-basis}, we can assume that
relations~\eqref{eweq_item}--\eqref{bordweq_item} are equalities.
Now, choosing $\BS$-alignments as in
Proposition~\ref{prop:classic-basis} for those $u_\varSU$,
$u'_\varSX$, $v_\varSZ$, $v'_\varSU$, $w_\varSX$ and $w'_\varSZ$
which are $\BS$-aligned with~$L$, one can define three homomorphisms
$\theta$ as in that proposition. One can construct a split exact
sequence at the level of Witt groups of~$\BS$, and compare it with
the localization long exact sequence for~$\CS$ and~$L$ via the
various~$\theta$. This is a morphism of long exact sequences by the
compatibilities of the lax product structure with extension of
support, restriction to an open and connecting homomorphism, see
Remark~\ref{rema:mock-les}. The Five Lemma then gives the result. We
leave the details to the reader.
\end{proof}

\begin{rema}
The same theorem holds with support, \ie replacing $\Wcoh(\CS)$ by
$\Wcoh_{\!Z'}(\CS)$, and consequently $\Wcoh_{\!Z}(\CS)$ by
$\Wcoh_{\!Z\cap Z'}(\CS)$ and $\Wcoh(U)$ by $\Wcoh_{\!Z'\cap U}(U)$.
\end{rema}

\begin{rema} \label{slicePic_rema}
The benefit of this theorem, together with the one on d\'evissage, is
that we can build a total basis on $\CS$ out of smaller ones on $\CS$
with support in $Z$ and on $U$. As in Remark~\ref{rema:chunk} and
for the same reasons, the injectivity assumption on $P \to
\Pic_\BS(U)/2$ is not really restrictive in actual computations.
\end{rema}

\begin{rema} \label{coherent_rema}
All this ``total'' formalism still holds in the non-necessarily
regular case with the following modifications. All schemes should be
noetherian and separated. The category $\SmPic_\BS$ should be
replaced by the category of $\BS$-schemes $\CS$ that have a
dualizing complex (not necessarily injectively bounded), with the
conditions on Picard groups left unchanged\,: the important point is
that two dualizing complexes always differ by tensoring by a line
bundle (and a shift). The Witt groups considered for such schemes
$\CS$ should be the {\rm coherent} Witt groups, and the formalism
will mimic how they behave as a (total) module over the {\em locally
free} Witt groups of $\BS$. In particular, the $\BS$-alignment of
definition~\ref{defi:X-aligned} should be replaced by an isomorphism
$\phi: M^{\otimes 2} \otimes \pi_\CS^*N \otimes K' \isoto K$ where
$M$ and $N$ are line bundles, as before, but $K$ and $K'$ are
dualizing complexes. Pull-backs of coherent Witt groups should only
be considered along flat morphisms preserving dualizing complexes
(\eg\ open embeddings in the localization sequence). Morphisms
involving pull-backs of locally free Witt groups can be considered
without restriction (\eg\ pull-backs from the base $\BS$).
\end{rema}

\bibliography{total_Witt}

\end{document}